\documentclass{article}
\usepackage{geometry}               
\usepackage[english]{babel}
\usepackage[utf8]{inputenc}
\usepackage{amssymb,amsmath,amsthm,amsfonts}
\usepackage{graphicx}
\usepackage[usenames,dvipsnames]{xcolor}
\usepackage{cancel}
\usepackage[colorlinks]{hyperref}
\usepackage{dsfont}
\usepackage{mathtools}

\usepackage[sc]{mathpazo}

\usepackage{tikz,pgfplots,pgfplotstable}
\usepgflibrary{arrows.meta}
\usetikzlibrary{patterns}
\usetikzlibrary{calc}
\usepgfplotslibrary{fillbetween}
\usepgfplotslibrary{patchplots}
\numberwithin{equation}{section}

\pgfplotsset{compat=1.17}

\tikzset{> = {Stealth[inset=0pt]}}

\newcommand{\N}{\mathbb{N}}

\newcommand{\R}{\mathbb{R}}

\newcommand{\sphere}{\mathbb{S}}
\newcommand{\from}{\colon}

\newcommand{\eps}{\varepsilon}

\newcommand{\loc}{\text{loc}}
\newcommand{\vfi}{\varphi}

\newcommand{\Id}{\mathrm{Id}}

\newcommand{\Om}{\Omega}

\DeclareMathOperator{\bari}{\rm bar}

\newcommand{\Hcal}{{\mathcal{H}}}

\def\XXint#1#2#3{{\setbox0=\hbox{$#1{#2#3}{\int}$ }
\vcenter{\hbox{$#2#3$ }}\kern-.6\wd0}}

\newcommand{\ind}[1]{\mathds{1}_{#1}}

\DeclareMathOperator{\dist}{dist}
\DeclareMathOperator{\diverg}{div}

\DeclareMathOperator{\diam}{diam}

\newtheorem{proposition}{Proposition}[section]
\newtheorem{theorem}[proposition]{Theorem}
\newtheorem{corollary}[proposition]{Corollary}
\newtheorem{lemma}[proposition]{Lemma}

\theoremstyle{definition}
\newtheorem{definition}[proposition]{Definition}
\newtheorem{remark}[proposition]{Remark}

\newcommand{\beq}{\begin{equation}}
\newcommand{\eeq}{\end{equation}}
\newcommand{\ben}{\begin{enumerate}}
\newcommand{\een}{\end{enumerate}}
\newcommand{\bit}{\begin{itemize}}
\newcommand{\eit}{\end{itemize}}
\newcommand{\dys}{\displaystyle}

\newcommand{\NTr}{\ell}

\DeclareMathOperator{\od}{\Lambda}

\newcommand{\IM}{I}
\newcounter{fakecounter}

\title{Asymptotic  location and shape of the optimal favorable region  in a  Neumann spectral problem}

\author{Lorenzo Ferreri, Dario Mazzoleni, Benedetta Pellacci and Gianmaria Verzini}

\begin{document}
\maketitle

\begin{abstract}
We complete the study  concerning the minimization 
of the positive principal eigenvalue associated with a weighted Neumann problem settled in a bounded regular domain $\Omega\subset \R^{N}$, $N\ge2$, for the weight varying in a suitable class of sign-changing bounded functions. 
Denoting with $u$ the optimal eigenfunction and with 
$D$ its super-level set, corresponding to the positivity set of the optimal weight, we prove that, as the measure of $D$ tends to zero, the unique maximum point of $u$, 
$P\in \partial \Omega$,  tends to a point of maximal mean curvature of $\partial \Omega$. 
Furthermore, we show that $D$ is the intersection with $\Omega$ of a $C^{1,1}$ nearly 
spherical set, and we provide a quantitative estimate of the spherical asymmetry, which decays like a power of the measure of $D$.

These results provide, in the small volume regime, a fully detailed answer to some long-standing 
questions in this framework.
\end{abstract}
\noindent
{\footnotesize \textbf{AMS-Subject Classification}}. 
{\footnotesize 49R05, 49Q10; 92D25, 35P15, 47A75.}

\section{Introduction}

Let $\Omega\subset\R^N$ be a (open and connected) domain, $D\subset\Omega$ be a measurable 
subset of positive measure, $\beta>0$ be a given constant, and let us consider the principal eigenvalue  
\begin{equation}\label{eq:def_lambda_beta_D}
\lambda(D)=\lambda(D,\Omega):= \inf \left\{
\int_\Omega |\nabla u|^2\,dx :  
u\in H^1(\Omega),\ 
\int_D u^2 - \beta \int_{\Omega\setminus D} u^2\,dx = 1\right\}, 
\end{equation}
associated to the indefinite weighted Neumann eigenvalue problem
\begin{equation}\label{eq:Cg}
\begin{cases}
-\Delta u=\lambda m_D u&\text{in }\Omega,\\
\partial_\nu u=0 &\text{on }\partial \Omega,
\end{cases}
\qquad
\text{where }m_D:=\ind{D} - \beta\ind{\Omega\setminus D}.
\end{equation}
Throughout this paper we mainly deal with a fixed domain $\Omega$, bounded and with regular boundary 
(at least $C^{3,\theta}$, for some $\theta>0$, although for a relevant part of our results $C^{2,1}$ is enough), and we 
omit the dependence of $\lambda$ on $\Omega$ in case no confusion arises. For such an $\Omega$, 
$\lambda(D,\Omega)$ is strictly positive, and achieved by a strictly positive non-constant 
eigenfunction, if and only if $m_D$ has negative average (thus avoiding $\lambda(D) = 0$, with constant eigenfunction). This condition translates in terms of the Lebesgue measure of $D$ as \begin{equation}\label{eq:cond_on_delta}
0< |D|=:\delta < \dfrac{\beta|\Omega|}{\beta+1}.
\end{equation}
For any such $\delta$, let us consider the shape optimization problem
\begin{equation}\label{eq:def_od}
\od(\delta)=\min\Big\{\lambda(D):D\subset \Om,\mbox{ measurable, }|D|=\delta\Big\}.
\end{equation}
It has been proved in \cite{ly} that $\od(\delta)$ is achieved for every $\delta$ enjoying 
\eqref{eq:cond_on_delta} by an optimal shape $D_\delta$, with associated positive eigenfunction 
$u_\delta\in C^{1,\alpha}(\overline{\Omega})$, normalized in $L^2(\Omega)$. In particular, $u_\delta$ satisfies \eqref{eq:Cg} with 
$\lambda = \od(\delta)$ and $D=D_\delta$, and moreover $D_\delta$ is a superlevel set of $u_\delta$, for a suitable choice of the level.

In this paper we pursue the analysis started in \cite{MPV3}, concerning the asymptotic location and shape 
of the optimal set $D_\delta$, in the small volume regime $\delta\to0$ 
(in which \eqref{eq:cond_on_delta} is always satisfied). In particular, we provide a 
complete answer to some questions which were left open in \cite{MPV3}.

Before recalling the results contained in \cite{MPV3} and outlining our contributions here, let us 
briefly describe one main motivation to investigate \eqref{eq:def_od}, which comes from population 
dynamics and is related to the optimal design of an habitat to enhance the chances of persistence of a 
species. 
Consider a population of density $u=u(x,t)$ which disperses in an insulated region $\Omega$ according to a Cauchy-Neumann reaction-diffusion model of logistic type:
\[
\begin{cases}
u_t - d\Delta u = m(x) u - u^2& x\in\Omega,\ t>0,\\
u = u_0 \ge 0 & x\in\overline{\Omega},\ t=0\\
 \partial_\nu u = 0 & x\in\partial\Omega,\ t>0
\end{cases}
\]
(see \cite{MR1014659,MR2214420}). In such a model, favorable and hostile zones of the 
heterogeneous habitat 
respectively correspond to positivity and negativity sets of the sign-changing weight $m\in L^\infty(\Omega)$, while 
the constant motility rate $d>0$ encodes the intensity of diffusion. It is well known (see e.g.\ 
\cite{MR2214420}) that persistence of the population, for every nontrivial initial datum $u_0$, 
is equivalent to the existence of a positive steady state, which in turn is equivalent to the 
strict negativity of the (non-weighted) principal eigenvalue $\tilde\lambda_1=\tilde\lambda_1(m,d)$ of the 
eigenvalue problem
\begin{equation}\label{eq:Bg}
\begin{cases}
-d\Delta u-m(x)u=\tilde\lambda u&\text{in }\Omega,\\
\partial_\nu u=0 &\text{on }\partial \Omega.
\end{cases}
\end{equation}
On the other hand, using the monotonicity of $\tilde\lambda_1(m,d)$ with respect to $d$, it is not 
difficult to see that, for every fixed habitat $m\in L^\infty(\Omega)$, there exists a threshold 
motility $d^*=d^*(m)\in[0,+\infty]$ such that
\[
\tilde \lambda_1(m,d) < 0 \qquad\iff\qquad d < d^*(m).
\]
In particular, since $d^*$ is defined by $\tilde\lambda_1(m,d^*(m))=0$, we obtain 
that the corresponding positive eigenfunction $\varphi^*$ satisfies
\[
\begin{cases}
-\Delta \varphi^* = \dfrac{1}{d^*(m)} m \varphi^* &\text{in }\Omega,\\
\partial_\nu \varphi^*=0 &\text{on }\partial \Omega,
\end{cases}
\]
so that $\lambda_1(m) := \dfrac{1}{d^*(m)}$ is the positive principal eigenvalue of \eqref{eq:Cg} 
(with $m$ instead of $m_D$).

Since the population persists if and only if $d<d^*(m)$, it is natural to be interested in 
maximizing such threshold, i.e. minimizing $\lambda_1(m)$, with respect to $m$ or other relevant 
parameters of the problem. In this direction, there is a large variety of contributions in the 
literature (see e.g. \cite{haro,beroro,MR3498523,mapeve,dipierro2021nonlocal, MR3771424, PePiSc} and references therein), for instance including other possible types of diffusion or addressing different but related 
optimization problems, also appearing in the context of composite membranes  (see \cite{MR1796024, 
MR2421158} and the review \cite{MAZARI2022401}). On the other hand, one of the most considered 
problem is the minimization of $\lambda_1(m)$ for $m$ in a class of 
weights $\mathcal M$ which fixes the total resources $\int_\Omega m$, as well as pointwise lower 
and upper bounds $-\beta \le m \le 1$, see e.g. \cite{MR1014659,ly,llnp}. 
In such a case, it has been shown in \cite{ly} that the infimum of $\lambda_1(m)$ is achieved by a bang-bang (piecewise constant) optimal weight $m = m_D$, as in \eqref{eq:Cg}, 
for some measurable set $D\subset\Omega$, which represents the favorable zone of the optimal habitat.  
Therefore, the optimal design problem for the survival threshold of the population reduces to 
\eqref{eq:def_od}, where $\delta$ is prescribed by the fixed average of $m$. Moreover, any 
optimal set $D_\delta$ is a superlevel set of the associated positive eigenfunction $u_\delta$, so 
that the location of $D_\delta$ is somewhat related to that of maximum points of $u_\delta$.

While the picture is nowadays completely clear in dimension $N=1$, not very much is known about 
the properties of $D_\delta$, $u_\delta$, in dimension $N\ge2$; several conjectures and open 
questions have been formulated in the literature \cite{MR1105497,MR2214420,ly,mapeve,MAZARI2022401} about the shape and location of $D_{\delta}$, also in the case of other boundary conditions. Up to our 
knowledge, the only results for general domains $\Omega$ are contained in \cite{MPV3}, in 
the small volume regime, inspired by techniques in the framework of concentration results for 
semilinear problems \cite{nitakagi_cpam}. More precisely, the results in \cite{MPV3} are obtained through a blow-up procedure, in connection with the limit problem
\begin{equation}\label{eq:defIM}
\begin{split}
\IM&:=\min\Big\{\lambda(A,\R^N_+):A\subset \R^N_+,\text{ measurable, }|A|=1\Big\}\\
&=\min\Big\{\lambda(A,\R^N):A\subset \R^N,\text{ measurable, }|A|=2\Big\},
\end{split}
\end{equation}
where $\R^N_+:=\R^N\cap\{x:x_N>0\}$ denotes the $N$-dimensional upper half-space. It has been shown in \cite[Sec. 2]{MPV3}, by 
reflection and symmetrization, that $\IM$ is achieved by a half ball, centered at the boundary of $\R^N_+$, with an associated 
radial and radially decreasing eigenfunction. Precisely, we have that
\begin{equation}\label{eq:lambdazero}
\IM = \lambda(B_{r_2}^+, \R^N_+) = \lambda(B_{r_2}, \R^N),
\end{equation}
where $B_{r_2}\subset\R^N$ denotes the ball of measure $2$, centered at the origin, 
and $r_2$ denotes its radius. Moreover such minimizer is unique up to translations along $\partial \R^N_+$, and in turn 
$\lambda(B_{r_2}^+, \R^N_+)$ is achieved by (the restriction to $\R^N_+$ of) 
$w\in H^1_{\text{rad}}(\R^N)\cap C^{1,1}(\R^N)$, solution of 
\begin{equation}\label{eq:lim_prob}
-\Delta w = \IM m w \quad \text{in } \R^N ,
\qquad\text{where }m:=\ind{B_{r_2}} - \beta 
\ind{\R^N\setminus B_{r_2}}.
\end{equation}
Here $w>0$ is radially symmetric and radially decreasing; 
as a matter of fact, $w$ is explicit in terms of Bessel functions, and it decays 
exponentially at infinity:
\begin{equation}\label{eq:decay_w}
|w(x)|+|\nabla w(x)| \sim C |x|^{-(N-1)/2} e^{-\sqrt{\IM\beta}|x|}\qquad
\text{as }|x|\to+\infty.
\end{equation}
In the next statement we collect the part of the results obtained in \cite{MPV3} which is relevant 
for the present discussion. To this aim, for any $P\in\partial\Omega$ we denote with $H_P$ the mean 
curvature of $\partial \Omega$ at $P$. 
\begin{theorem}[{\cite[Thms. 1.2, 1.6]{MPV3}}]\label{th:intro_D}
Let $\partial\Omega$ be of class $C^{2,1}$.
There exists $\delta_{0}>0$ such that, for every $\delta\in (0,\delta_{0})$:
\begin{enumerate}
\item $u_\delta$ has a unique local maximum point $P_\delta\in \partial \Omega$; 
\item $D_\delta$ is connected.
\setcounter{fakecounter}{\theenumi}
\end{enumerate}
Moreover, as $\delta \to0$,
\begin{enumerate}
\addtocounter{enumi}{\value{fakecounter}}
\item for any choice of $r_\pm(\delta)$ such  that $|B_{r_\pm(\delta)}|=2\delta(1\pm o_\delta(1))$,
\[
B_{r_-(\delta)}(P_\delta)\cap\Omega \; \subset \; D_\delta \;\subset\; B_{r_+(\delta)}(P_\delta)
\cap\Omega;
\]
\item\label{punto4} for a universal constant $\Gamma>0$, explicit in terms of $w$ (see equation \eqref{eq:Gamma}),  
\[
\IM \delta^{-2/N}\left(1+o_\delta(1)\right)\le \od(\delta)\leq \IM \delta^{-2/N}\left(
1-\Gamma  H_{P}\delta^{1/N}+o(\delta^{1/N})\right),
\] 
for any $P\in \partial \Omega$.
\end{enumerate}
\end{theorem}

In particular, for $\delta$ small, the optimal set $D_\delta$ is connected, it intersects 
$\partial\Omega$, and it roughly looks as the intersection of $\Omega$ with a shrinking ball centered at 
moving points $P_\delta\in\partial\Omega$. In this respect, the main questions left open in 
\cite{MPV3} concern the asymptotic location of $P_\delta$ as $\delta\to0$, as well as more quantitative 
information about the asymptotic shape of $D_\delta$. 

Our aim here is twofold. First of all, we are now able to obtain an exact expansion of the optimal 
eigenvalue with respect to $\delta$, which allows us to detect the location of $P_{\delta}$ at the points of highest mean curvature of $\partial \Omega$, as stated in our first main result.
\begin{theorem}\label{thm:sviluppoesatto}
With the same assumptions and notations of Theorem \ref{th:intro_D}, we have that, as $\delta\to0$, 
\begin{enumerate}
\item
$\dys H_{P_\delta}\rightarrow \max\{H_P : P\in \partial\Omega\}=:\widehat H;$
\item
$\dys \od(\delta)= \IM \delta^{-2/N}\left(1-\Gamma \widehat H \delta^{1/N}+o(\delta^{1/N})\right).$
\end{enumerate}
\end{theorem}
%
The proof of Theorem \ref{thm:sviluppoesatto} is based on the 
asymptotic expansion of a suitable Rayleigh quotient related to $\od(\delta)$. In performing this expansion one cannot take advantage of any kind of linearization argument, due to the presence of the discontinuous weight. However,
exploiting the analogies with concentration results for singularly perturbed semilinear elliptic equations, obtained in \cite{MR1736974, delpino_flores}, we  can show a refined 
exponential decay of $u_{\delta}$. This allows us to obtain the desired second order expansion of $\od(\delta)$ yielding the information in Theorem  \ref{thm:sviluppoesatto}.

Next, we focus on the asymptotic shape of the optimal set
$D_{\delta}$ and we show that it is nearly spherical in a quantitative way, as in the following  result.
\begin{theorem}\label{thm:nearlyspher1}
Let $\partial\Omega$ be of class $C^{3,\theta}$, for some $\theta>0$.
There exists $\overline{\delta}>0$ such that, for every $\delta \in (0,\overline{\delta})$, there exists 
$Q_\delta\in\partial\Omega$, $\rho_{\delta}\in C^{1,1}({\mathbb S}^{N-1})$, such that
\begin{equation}\label{eq:defrho}
D_{\delta}=\left\{x\in\Omega : |x-Q_{\delta}|<\delta^{1/N}\left(r_{2}+\rho_{\delta}\left(\frac{x-Q_{\delta}}{|x-Q_{\delta}|}\right)\right)\right\}
\end{equation}
where $r_2$ is the radius of the ball of measure $2$.

Moreover, as $\delta\to0$, $\delta^{-1/N}|P_\delta-Q_\delta|\to0$ and
\begin{enumerate}
\item
$\|\rho_{\delta}\|_{C^{1,1}}\to 0$;
\item
$\|\rho_{\delta}\|_{L^2} = o(\delta^{1/2N})$.
\end{enumerate} 
\end{theorem}

The proof of this theorem is based on the expansion of $\od(\delta)$ obtained in Theorem 
\ref{thm:sviluppoesatto} and on sharp quantitative estimates about problem \eqref{eq:defIM}, 
adapted from \cite{ferreri_verzini2}, 
where the analogous problem is considered in the case of Dirichlet boundary conditions. 
Such adaptation is highly nontrivial and technically demanding: in the Neumann case concentration happens at 
the boundary, and quantitative estimates are strongly affected by the local regularity and geometry of 
$\partial\Omega$.
In particular, 
the point $Q_\delta$ can be seen as a sort of projection on $\partial\Omega$ of the barycenter of 
$D_\delta$. Its precise definition is provided in Lemma \ref{lem:bari} ahead, and its role is to 
prevent translations of optimizers associated with \eqref{eq:defIM}. Moreover, we remark that in 
principle the eigenfunctions $u_\delta$, for which $D_\delta$ is a superlevel set, are naturally 
uniformly bounded only in $C^{1,\alpha}$, $\alpha<1$; the $C^{1,1}$ control in the above theorem is 
more delicate to be obtained, and it requires the use of recent results about the regularity of 
transmission problems, see \cite{Caffarelli2021:TransmissionProblems,Dong2021:TransmissionProblems}. 
Actually, this is the only part of the argument which requires the further regularity assumptions on $\partial\Omega$; on the other hand, such further regularity should reflect also on that of the 
free boundary, see Remark \ref{rmk:piùregolare}.
Furthermore, by Gagliardo-Nirenberg inequality, 
it is possible to combine the estimates in Theorem \ref{thm:nearlyspher1} to obtain quantitative information concerning the rate of decay of the $C^{1,\alpha}$ norm of $\rho_{\delta}$, for $\alpha<1$ 
(see Remark~\ref{rmk:GagliardoNirenberg}):
\[
\|\rho_\delta\|_{C^{1,\alpha}}=o\left(\delta^{\frac{(1-\alpha)}{N(4+N)}}\right) \quad \forall\, \alpha\in (0,1).
\]

In view of Theorems \ref{thm:sviluppoesatto}, \ref{thm:nearlyspher1}, we have a fairly 
complete picture about the shape and location of the optimal favorable shape $D_\delta$, in 
the small volume regime, in the case of Neumann boundary conditions. As we mentioned, in the 
case of Dirichlet boundary conditions, the same issues have been recently faced in 
\cite{ferreri_verzini2}: also in this case the optimal shape is $C^{1,1}$ nearly spherical, 
but asymptotically located at a point inside $\Omega$ with maximal distance from the boundary. 
As far as the asymptotic location and shape is concerned, these 
results completely agree with those in dimension $N=1$: it is well known that, 
for any $\delta$, the optimal shape $D_\delta$ is an interval positioned at the boundary of 
the interval $\Omega$ (for Neumann boundary conditions) or at its center (for Dirichlet ones). 
From this perspective, these results also provide positive answers to questions raised in the 
literature \cite{MR1105497,ly,MAZARI2022401}, which in turn were partially motivated by the above 
mentioned one-dimensional description. 

On the other hand, the quantitative estimates of the spherical asymmetry show the emergence of a 
phenomenon which is peculiar to the case of dimension $N\ge2$. Indeed, in dimension $N=1$, both 
$\Omega$ and $D_\delta$ are intervals, i.e. balls, with no spherical asymmetry. In dimension $N\ge2$, 
this is possible only if $\Omega$ is a ball and Dirichlet boundary conditions are assumed, see 
\cite{llnp}. For general $\Omega$, or for Neumann conditions, the spherical asymmetry 
is non-trivial, and it exhibits very different decay rates. Indeed, in the Dirichlet case,  
the decay fully inherits that of the solution of the limit problem, namely  of exponential 
rate; in the case of Neumann boundary conditions, it is triggered by the geometry and the regularity 
of the boundary $\partial\Omega$. This phenomenon enlighten once again analogies with the semilinear 
case.

{\bf Structure of the paper.}  Section~\ref{sec:expdecay} is devoted to the proof of the exponential decay of the eigenfunction $u_\delta$ and of $w_\delta$, its counterpart in the blow-up configuration.
Thanks to this information, in Section~\ref{sec:boundbelow} we can prove Theorem~\ref{thm:sviluppoesatto}, in particular the bound from below in the expansion of $\od(\delta)$.
Section~\ref{sec:polar} is devoted to prove that $D_\delta$ is nearly spherical and to show that its parametrization is actually $C^{1,1}$, exploiting the regularity results for transmission problems.
Finally, in Section~\ref{sec:quantitative} we use the quantitative estimates from~\cite{ferreri_verzini2} to conclude the proof of Theorem~\ref{thm:nearlyspher1}.


\bigskip
\textbf{Notation.}
\begin{itemize}
\item $|\cdot|$ denotes the Lebesgue $N$ dimensional measure and $\Hcal^{N-1}(\cdot)$ the Hausdorff $N-1$ dimensional measure.
\item For a function $f$, its positive/negative parts are denoted as 
$f^\pm(x)=\max\{\pm f(x) , 0\}$.
\item The characteristic function of a set $E$ is denoted by $\ind{E}$. 
\item $B_r(x)$ denotes the ball of radius $r>0$ centered at $x\in \R^N$. If $x=0$, we often write 
$B_r=B_r(0)$. We call $\omega_N = |B_1|$ the measure of a ball of radius $1$.
\item We denote with $r_{2}>0$ the radius of the ball $B_{r_{2}}$, such that $|B_{r_{2}}|=2$.
\item $\R^{N}_+ = \R^{N-1}\times \R_+$, $B^{+}_{r}=B_{r}\cap (\R^{N}_+)$.
\item $H_P$ denotes the mean curvature of $\partial \Omega$ at $P\in\partial\Omega$, and $\widehat H = \max_{P\in\partial\Omega} H_P$.
\item  Given a set $A\subset \R^N$, we denote $\bari(A)$ its barycenter.
\item $C,C_1,C',\dots$ denote any (non-negative) universal constant, which may also change 
from line to line.
\end{itemize}

\section{Uniform exponential decay of the eigenfunctions}\label{sec:expdecay}

Let us recall that, for $\delta>0$ small, $u_\delta$ denotes the positive, $L^2$-normalized  
eigenfunction associated with $\od(\delta)=\lambda(D_\delta,\Omega)$, with unique maximum point 
$P_\delta\in\partial D_\delta \cap \partial \Omega$. 

The starting point of this section is the following version of a result which was obtained in \cite{MPV3}.

\begin{theorem}[{\cite[Proposition~4.11]{MPV3}}]\label{thm:2.3}
With the notation above, for all $\eta>0$ there exist $\delta_0>0$, $\widehat C>0$ such that, 
for all $\delta  \in  (0, \delta_0)$, there is a subdomain $\Omega_\delta^{(i)}\subset \Omega$  satisfying:
\begin{enumerate}
\item[(i)] $P_\delta\in \partial \Omega_\delta^{(i)}$ and $\diam(\Omega_\delta^{(i)})\leq \widehat C\delta^{1/N}$,
\item[(ii)] $\|u_\delta\|_{L^\infty(\Omega)}\leq C^*\delta^{-1/2}$,
\item[(iii)] $|u_\delta(x)|\leq C^* \delta^{-1/2}\eta e^{-\frac{\mu_1d(x)}{\delta^{1/N}}}$, for all $x\in \Omega\setminus \Omega_\delta^{(i)}$,
\end{enumerate}
where $d(x):=\min\{\dist(x,\partial \Omega_\delta^{(i)}),\eta_0\}$ and $C^*, \mu_1,\eta_0$ are positive 
constants depending only on $\Omega$.
\end{theorem}

With a careful look at the proof of~\cite[Theorem~2.3]{nitakagi_cpam} and using the same approach of~\cite[Lemma~2.3]{delpino_flores}, we can prove the following decay result for $u_\delta$.

\begin{theorem}\label{thm:expdecayudelta}
There exist universal constants $C_1,C_2>0$ and $\delta_0>0$ such that,  for all $0<\delta< \delta_0$,
\begin{equation}\label{eq:expdecayudelta}
u_\delta(x)\leq C_1 \delta^{-1/2} e^{-C_2|x-P_\delta|\delta^{-1/N}},\qquad  \text{for all }x\in \Omega.
\end{equation}
\end{theorem}
\begin{proof}

{\it Step 1. For all $\eps>0$ there is $R=R(\eps)>0$ such that, for $\delta$ sufficiently small we have $\delta^{1/2}|u_\delta(x)|\leq  \eps$ for $|x
-P_\delta|>R\delta^{1/N}$.} 

This follows thanks to Theorem~\ref{thm:2.3}, choosing $\eta=\tfrac{\eps}{C^*}$. In fact, thanks to $(i)$, we deduce that if $|x-P_\delta|>R\delta^{1/N}$ for $R>2\widehat C$ (independent of $\delta$), then $x\in \Omega\setminus \Omega_\delta^{(i)}$.
Eventually, thanks to $(iii)$, we conclude the claim, since the exponential is for sure less than or equal to $1$.

{\it Step 2. There exist $R_0>0$ and $\nu_0>0$ such that for all $R>R_0$ and $\delta$ sufficiently small we have \[
\sup_{|x-P_\delta|>R\delta^{1/N}}\delta^{1/2}u_\delta(x)\geq 2\sup_{|x-P_\delta|>(R+\nu_0)\delta^{1/N}}\delta^{1/2}u_\delta(x).
\]
}

Let us assume, for the sake of contradiction, that there are sequences $R_n\to+\infty$, $\nu_n\to+\infty$, $\delta_n\to 0$ and $x_n\in \Omega$ such that $|x_n-P_n|\geq (R_n+\nu_n)\delta^{1/N}$, being $P_n\in \partial\Omega$ the maximum point for $u_{\delta_n}$ and \[
\delta_n^{1/2}u_{\delta_n}(x_n)=\mu_n>\frac{M_n}{2},\qquad M_n=\sup_{|x-P_n|>R_n\delta_n^{1/N}}\delta_n^{1/2}u_{\delta_n}(x).
\]

For the uniform decay proved in Step 1, we deduce that $\mu_n,M_n\to 0$ as $n\to+\infty$.
We define the auxiliary function\[
v_n(y)=\delta_n^{1/2}\frac{u_{\delta_n}(\delta_n^{1/N}y+x_n)}{\mu_n},
\]
then $v_n(0)=1$, $0< v_n< 2$ if  $|y|<\nu_{n}$ (thus $|\delta_n^{1/N}y+x_n-P_n|> R_n\delta_n^{1/N}$)
and $v_n$ solves the equation\[
\begin{cases}
-\Delta v_n(y)=\delta_n^{2/N}\od(\delta_n) m_{\delta_n}(\delta_n^{1/N}y+x_n)v_n(y),\qquad &\text{for }y\in \frac{\Omega-x_n}{\delta_n^{1/N}},\\
\partial_\nu v_n=0,\qquad &\text{on }\partial \frac{\Omega-x_n}{\delta_n^{1/N}}.
\end{cases}
\]
For every compact set $K\subset \R^N$,  $m_{\delta_n}=-\beta$ for $n$ sufficiently big in $K\cap (\delta_n^{-1/N}(\Omega-x_n))$.
Then we can use Theorem \ref{th:intro_D} to pass to the limit as $n\to+\infty$; we obtain that 
$v_n\to v$ locally uniformly (and locally $W^{2,p}$) to a positive solution to either 
\[
\begin{cases}
-\Delta v=-\IM\beta  v,\qquad &\text{in }\R^{N}_+,\\
\partial_\nu v=0,\qquad &\text{on }\R^{N-1},
\end{cases}
\qquad \text{or}\qquad 
-\Delta v=-\IM \beta v,\qquad \text{in }\R^N.\\
\]
This gives a contradiction, because the only bounded nonnegative solution of both the above problems is the trivial one, while $v(0)=1$.

{\it Conclusion.} Iterating Step $2$ one obtains that, for all $k\in \N$ (taking $\nu_0>R_0$, which is always possible), \[
\sup_{|x-P_\delta|>k\nu_0\delta^{1/N}}u_\delta(x)\leq 2^{-k}\sup_{|x-P_\delta|>\nu_0\delta^{1/N}}u_\delta(x)\leq 2^{-k}\|u_\delta\|_{L^\infty(\Omega)}.
\]
For all $x\in \Omega$, we can find $k\in \N$ such that \[
k\leq \frac{|x-P_\delta|}{\nu_0\delta^{1/N}}\leq k+1;
\]
recalling also  conclusion (ii) of Theorem \ref{thm:2.3}, we obtain
\[
u_\delta(x)\leq 2^{-k}C^*\delta^{-1/2}\leq C\delta^{-1/2}e^{-\frac{|x-P_\delta|}{\nu_0\delta^{1/N}}},
\]
which in turn yields~\eqref{eq:expdecayudelta}.
\end{proof}

As a consequence, we have the following estimates.

\begin{corollary}\label{cor:decayint}
There exist $\delta_0, R_0>0$ and universal constants $C_1,C_2$ such that, for all $R>R_0$ and $\delta\leq \delta_0$ we have 
\begin{equation}\label{eq:decayint}
\begin{split}
\int_{\{x\in \Omega :\, |x-P_{\delta}|>R\delta^{1/N}\}}|u_\delta|^2\leq  C_1 \delta^{-1} e^{-2C_2R},\\
\int_{\{x\in \Omega :\, |x-P_{\delta}|>R\delta^{1/N}\}}|\nabla u_\delta|^2 \leq  C_1 \delta^{-1} e^{-2C_2R}.
\end{split}
\end{equation}
\end{corollary}
\begin{proof}
The first estimate is immediate. Indeed, in view of Theorem \ref{thm:expdecayudelta}
\[
\int_{\{x\in \Omega :\, |x-P_{\delta}|>R\delta^{1/N}\}}|u_\delta|^2\leq  |\Omega|C_1 \delta^{-1} e^{-2C_2R}
.
\]

We aim now to provide a decay information on the $L^2$ norm of the gradient of $u_\delta$. 
It is clear that for all $R>3$,  ${\{x\in \Omega :\, |x-P_{\delta}|>R\delta^{1/N}\}}\subset D_{\delta}^{c}$, so that the equation satisfied by $u_\delta$ becomes 
\[
-\Delta u_\delta=-\beta \od(\delta)u_\delta.
\]
We take a smooth cutoff function so that 
\[
\begin{cases}
\eta=1,\qquad &\text{in }\Omega\setminus B_{2R\delta^{1/N}},\\
|\nabla \eta|\leq 1,\qquad &\text{in }\Omega \cap (B_{2R\delta^{1/N}}\setminus B_{R\delta^{1/N}}),\\
\eta=0,\qquad &\text{in } B_{R\delta^{1/N}},
\end{cases}
\]
which is possible up to take $R_0$ sufficiently big (here all the balls are centered at $P_{\delta}$). Then, testing the equation with $\eta^2u_\delta\in H^{1}(\Omega)$, we obtain 
\[
\int_{\Omega\setminus B_{R\delta^{1/N}}} \nabla u_\delta\cdot \nabla (u_\delta \eta^2)=-\beta\od(\delta)\int_{\Omega\setminus B_{R\delta^{1/N}}}\eta^2u_\delta^2\leq 0,
\]
thus\[
\int_{\Omega\setminus B_{R\delta^{1/N}}}|\nabla (u_\delta\eta)|^2-\int_{\Omega\setminus B_{R\delta^{1/N}}}|\nabla \eta|^2u_\delta^2=\int_{\Omega\setminus B_{R\delta^{1/N}}}|\nabla u_\delta|^2\eta^2+2\int_{\Omega\setminus B_{R\delta^{1/N}}}u_\delta\eta\nabla \eta\cdot \nabla u_\delta\leq 0.
\]
As a consequence, using~\eqref{eq:expdecayudelta}, we have \[
\begin{split}
\int_{\Omega\setminus B_{2R\delta^{1/N}}}|\nabla u_\delta|^2&\leq \int_{\Omega\setminus B_{R\delta^{1/N}}}|\nabla (u_\delta\eta)|^2
\leq \int_{ B_{2R\delta^{1/N}}\setminus B_{R\delta^{1/N}}}|\nabla \eta|^2u_\delta^2\leq C_{1}  \delta^{-1} e^{-2C_2 R},
\end{split}
\]
so that also the second part of the claim is proved.
\end{proof}

\section{Sharp bound from below with the curvature}\label{sec:boundbelow}

This section is devoted to the proof of Theorem \ref{thm:sviluppoesatto}.
We first recall that the constant $\Gamma$ in the expansion has been calculated in \cite{MPV3}, 
in terms of the eigenfunction $w$ of the 
limit problem \eqref{eq:defIM}, see \eqref{eq:lim_prob}. With the present notation, it reads
\begin{equation}\label{eq:Gamma}
\Gamma = \frac{2(N-1)\gamma}{\int_{\R^{N}_+}|\nabla w|^2\, dz},
\qquad\text{where }
\gamma = \frac{1}{N+1}\int_{\R^{N}_+}|\nabla w|^2z_N \,dz
\end{equation}
(here, for $z\in\R^N_+$, we write $z=(z',z_N)$, with $z'\in \R^{N-1}$ and $z_N>0$).

We are going to prove the following result.
\begin{theorem}\label{thm:boundfrombelow}
Let $P_\delta\in \partial \Omega$ be the unique maximum point of the function $u_\delta$ attaining 
$\od(\delta)$, and let us define
\begin{equation}\label{eq:alpha}
\alpha_\delta=(N-1)H_{P_\delta},
\end{equation}
where $H_{P_\delta}$ denotes the mean curvature of $\partial \Omega$ at $P_\delta$.

We have, as $\delta\to0$, 
\[
\od(\delta)\geq 
\IM\,\delta^{-2/N}\left(1-\frac{2\gamma\;\alpha_\delta}{\int_{\R^{N}_+}|\nabla w|^2}\;\delta^{1/N}+o(\delta^{1/N})\right).
\]
\end{theorem}

As a matter of fact, Theorem \ref{thm:sviluppoesatto} follows at once from the result above.
\begin{proof}[Proof of Theorem \ref{thm:sviluppoesatto}]
Recalling Theorem \ref{th:intro_D}, point \ref{punto4}, and since   
\begin{equation}\label{eq:Gamma1}
\Gamma H_{P_\delta} = \frac{2\gamma\;\alpha_\delta}{\int_{\R^{N}_+}|\nabla w|^2},
\end{equation}
we infer that Theorem  \ref{thm:boundfrombelow} implies
\[
H_{P_\delta} + o_\delta(1) \ge H_P,\qquad \text{ for every }P\in\partial\Omega, 
\]
as $\delta \to0$, and both claims in Theorem \ref{thm:sviluppoesatto} follow.
\end{proof}

The remaining part of this section is devoted the proof of Theorem~\ref{thm:boundfrombelow}, which is 
rather long and needs many intermediate steps. Such proof is based on an improved analysis of the 
blow-up procedure used in \cite[Section 4]{MPV3}, which in turn was inspired by \cite{nitakagi_cpam}. 
We summarize the key points of such procedure in the following. 

Recall that the domain $\Omega\subset \R^N$ is at least $C^{2,1}$, and let 
$P\in \partial \Omega$ to be chosen below. We call $x=(x_1,\dots,x_N)$ a set of coordinates centered at $P$, translated so that $P$ is the origin and rotated so that the outer unit normal to the boundary 
of $\Omega$ at $P$ is $-e_N$. 
Using the notation\[
x'=(x_1,\dots,x_{N-1}),
\]
there exist $d_0>0$, a $C^{2,1}$ function 
\begin{equation}\label{eq:psidelta}
\psi\colon \Big\{x'\in \R^{N-1}: |x'|<d_0\Big\}\to \R,
\end{equation}
and a neighborhood of the origin $\mathcal N$ such that 
\begin{enumerate}
\item[i)]$\psi(0)=0$, $\nabla \psi(0)=0$, $\Delta \psi(0)=(N-1)H_0 =: \alpha$,
\item[ii)] $\displaystyle
\partial \Omega\cap \mathcal N=\Big\{(x',x_N) : x_N=\psi(x')\Big\},\qquad \Omega\cap \mathcal N=\Big\{(x',x_N) : x_N>\psi(x')\Big\}.
$
\end{enumerate}
For a certain $d_1>0$, we define a diffeomorphism \[
\Phi\colon \Big\{y\in \R^N : |y|\leq d_1\Big\}\to \R^N,\qquad x=\Phi(y)=(\Phi_1(y),\dots, \Phi_N(y)),
\]
as 
\[
\Phi_j(y)=
\begin{cases}
y_j-y_N\frac{\partial \psi}{\partial y_j}(y'),&\qquad \text{for }j=1,\dots, N-1,
\smallskip\\
y_N+\psi(y'),&\qquad \text{for }j=N.
\end{cases}
\]
\begin{remark}\label{rmk:unoinmeno}
It is worth noticing that, in case $\partial\Omega$ is of class $C^{k,\alpha}$, then $\Phi$ is only 
$C^{k-1,\alpha}$. In particular, under our assumptions, $\Phi$ is always at least $C^{1,1}$.
\end{remark}
We note that $D\Phi(0)=\Id$, due to the properties of $\psi$, and therefore $\Phi$ is locally invertible in, say, 
$B_{3\NTr}$ for some $\NTr>0$.  Then we can assume 
\begin{equation}\label{eq:defPsi}
\Phi(B_{2\NTr}^+)\subset\Omega,\quad\text{and }\quad \Psi\from \Phi(B_{3\NTr}^+)\to B_{3\NTr}^+,\ \Psi(x):=\Phi^{-1}(x).
\end{equation}

The map $\Psi$ can be seen as a local diffeomorphism straightening the boundary around $0\in \partial \Omega$. For future reference, we remark that
\begin{equation}\label{eq:lemmaA1}
\begin{split}
\det D\Phi (y)  &= 1 -\alpha y_N + O(|y|^2),\smallskip\\
\left| \frac{y}{|y|}D\Psi (\Phi(y))\right|^2  &= 1 +2 y_N \sum_{i,j=1}^{N-1}  
\psi_{ij}(0)\frac{y_iy_j}{|y|^2} + O(|y|^2),
\end{split}
\qquad\text{as  }y\to0,
\end{equation}
where $\psi_{ij}=\frac{\partial \psi}{\partial y_{j}\partial y_{i}}$
and we refer to \cite[Lemma A.1]{nitakagi_cpam} for further details.

In our case, we choose $P=P_{\delta}$, the maximum point of the optimal eigenfunction $u_\delta$. 
As a consequence, the rotation and translation to set 
$P_{\delta}$ at the origin become $\delta-$dependent, and the Taylor expansions in \eqref{eq:lemmaA1} 
hold with $\alpha=\alpha_\delta$,  $\psi_{ij}=\psi^{\delta}_{ij}$. 
Let us also observe that all the decay estimates of the previous section will be applied in this one taking $P_{\delta}=0$.

The transformed eigenfunction is defined by 
\begin{equation}\label{eq:vk}
v_\delta(y):=u_{\delta}(\Phi_\delta(y)),\qquad y\in  \overline B^{+}_{2\kappa}.
\end{equation}
It is then easy to extend by symmetry in the whole $B_{2\kappa}$ the function $v_k$ also where $y_N<0$:
\begin{equation}\label{eq:vtildek}
\widetilde v_\delta(y):=
\begin{cases}
v_\delta(y),\qquad &\text{if }y_N\geq 0,\\
v_\delta(y',-y_N),\qquad &\text{if }y_N<0.
\end{cases}
\end{equation}
At this point, we can introduce the blow-up sequence, for $\delta_k>0$,
\begin{equation}\label{eq:wdelta}
w_\delta(z)=\delta^{1/2}\,\widetilde v_\delta\left(\delta^{1/N}z\right),\qquad z\in \overline{B_{\kappa\delta^{-1/N}}}.
\end{equation}
We obtain that $w_\delta$ satisfies the following equation in divergence form:
\begin{equation}\label{eq:eq_in_forma_di_div}
\begin{cases}
-\diverg\left( A^{\delta} \nabla  w_{\delta} \right) = \delta^{2/N} \od(\delta)J_{\delta}\cdot \widetilde m_\delta   w_{\delta}   & \text{in }  B_{\kappa\delta^{-1/N}} , \\
\partial_N  w_{\delta} = 0 = A^{\delta} \nabla  w_{\delta} \cdot e_N & \text{on }
 \{z_N = 0 \} \cap B_{\kappa\delta^{-1/N}} ,
\end{cases}
\end{equation}
where the rescaled weight is, for $z\in B_{\kappa \delta^{-1/N}}$,
\[
\widetilde m_\delta(z)=
\begin{cases}
m_{\delta}\Big(\Phi_\delta(\delta^{1/N}z)\Big),\qquad &\text{if }z_N\geq 0,\\
m_{\delta}\Big(\Phi_\delta(\delta^{1/N}z',-\delta^{1/N}z_N\Big),\qquad &\text{if }z_N< 0.
\end{cases}
\]
and the scalar $J_{\delta}$ and the matrix $A^\delta$ are defined as
\[
J_\delta(z) = |\det D\Phi_\delta (\delta^{1/N}z)|, \qquad A^\delta(z) = J_\delta(z)
[D\Psi_\delta(\Phi_\delta(\delta^{1/N}z))]^T D\Psi_\delta(\Phi_\delta(\delta^{1/N}z))
\]
for $z_N\ge 0$, and extended in the natural way for $z_N<0$. 
\begin{remark}\label{rmk:non_div_form}
Notice that the coefficients matrix $A^\delta$ is Lipschitz continuous. Moreover, 
in case $\partial\Omega\in C^{3,\theta}$, we have that $A^\delta$ is $C^{1,\theta}$ 
with respect to $z'$, as the even reflection involves only $z_N$ (recall Remark \ref{rmk:unoinmeno}).
In any case, it is standard to see that 
\eqref{eq:eq_in_forma_di_div} can be written also in non-divergence form, with a drift 
term with $L^\infty$ coefficients, see 
e.g.~\cite[Theorem~8.8]{gilbargtrudinger}. In particular, since 
$\tfrac{\partial w_\delta}{\partial z_N}=0$ on $\{z_N=0\}$, we have that 
$w_\delta$ is a $W^{2,p}(B_{\kappa\delta^{-1/N}})$ solution of such non-divergence 
form equation. We refer to~\cite[Section~4]{MPV3} for further details.
\end{remark}

Under the above construction, by the results in~\cite{MPV3}, Section~4 (in particular, Remark 4.4 
therein), we obtain the following convergence properties for the blow-up sequences. 
\begin{proposition}\label{prop:vecchiolavoro}
Under the above notation, and denoting with $w$, $m$ the optimizers of the limit problem, as in 
\eqref{eq:lim_prob}, we have
\begin{itemize}
\item 
$ w_\delta \to w$ strongly in $H^1_{\loc}(\R^N)$ and in $C^{1,\alpha}_{\loc}$;
\item $\widetilde m_{\delta} \stackrel{*}{\rightharpoonup} m$ weakly$*$ in $L^\infty_{\loc}$.
\end{itemize}
as $\delta\to0$.
\end{proposition}

The exponential decay of $u_\delta$, obtained in Corollary~\ref{cor:decayint}, entails an 
(uniform in $\delta$) exponential decay for $w_\delta$ and its gradient.
\begin{lemma}\label{lem:expdecayw}
There exist universal positive constants $C_3,C_4>0$ and there exist $R_0,\delta_0>0$ such that, for all $R>R_0$ and $\delta\leq \delta_0$ we have that
\begin{equation}\label{eq:expdecaywdelta}
\begin{split}
w_\delta(z)&\leq C_3 e^{-C_4|z|},\qquad  \text{for all }z\in B_{\kappa\delta^{-1/N}}\;,\\
\int_{B_{\kappa\delta^{-1/N}} \setminus B_R}|w_\delta|^2&\leq  C_3 e^{-2C_4R},
\end{split}
\end{equation}
Moreover, for $R_2>R_1>R_0$,
\begin{equation}\label{eq:decaygradw}
\begin{split}
|\nabla w_\delta(z)|&\leq C_3 e^{-C_4 |z| },\qquad  \text{for all }z\in B_{\kappa\delta^{-1/N}}\;,\\
\int_{B_{\kappa\delta^{-1/N}} \setminus B_R}|\nabla w_\delta|^2 &\leq  C_3 e^{-2C_4R},
\end{split}
\end{equation}
where the balls are centered at the origin, maximum point for $w_\delta$.
\end{lemma}
\begin{proof}
The pointwise exponential decay of $w_\delta$ (and of its $L^2$ norm) follows directly from 
Theorem \ref{thm:expdecayudelta}  and the fact that the diffeomorphism is close to the identity and centered at $P_{\delta}$ (which is translated at $0$), 
so that 
\[
0<w_\delta(z)=\delta^{1/2}u_\delta(\Phi_{\delta}(\delta^{1/N}z))\leq 
C_1e^{-C_2\delta^{-1/N}|\Phi_{\delta}(\delta^{1/N}z)|}
\leq C_3 e^{-C_4|z|}.
\] 
Then the decay of its $L^2$ norm outside $B_R$ is immediate.

Concerning the exponential decay of the gradient, we use elliptic regularity, see 
e.g.~\cite[Theorem~9.11]{gilbargtrudinger} (recall Remark \ref{rmk:non_div_form}) together with the 
Morrey-Sobolev embedding of $W^{2,p}\subset C^{1,\alpha}$ for $p$ sufficiently large.
More precisely, let $z$ be such that $B_{4}(z) \subset B_{\kappa\delta^{-1/N}}$. Then, 
using \eqref{eq:expdecaywdelta}, we have that, 
\[
\|w_\delta\|_{C^{1,\alpha}(B_{3}(z)\setminus B_{2}(z))}\leq C\|w_\delta\|_{W^{2,p}(B_{3}(z)\setminus B_{2}(z))}\leq C \|w_\delta\|_{L^p(B_{4}(z)\setminus B_{1}(z))}\leq C  e^{-C_4|z|}
\]
for a universal constant $C>0$, and
where we have taken into account that the balls are centered at a generic point $z$. 
\end{proof}

%
\begin{remark}\label{rmk:conv_forte}
In view of Lemma \ref{lem:expdecayw}, one can use a concentration-compactness kind of argument 
to obtain strong $H^1(\R^N)$ convergence of $w_\delta$ to $w$. We refer to 
\cite[Lemma 2.3]{ferreri_verzini2}, \cite[Section 4]{ferreri_verzini}, where  the argument 
was fully detailed in the case of Dirichlet boundary conditions.  
\end{remark}

Going back to the blow-up procedure, we now introduce the optimal sets (nonrescaled and rescaled) 
\begin{equation}\label{eq:Dtilde}
\begin{split}
D_{\delta}=&\{x\in \Omega : m_{\delta}(x)=1\},\qquad |D_{\delta}|=\delta,
\\
\widetilde D_\delta:=&\left\{z\in B_{\kappa\delta^{-1/N}} : z\in \frac{\Psi_\delta(D_{\delta})}{\delta^{1/N}}\text{ or }(z',-z_N)\in \frac{\Psi_\delta(D_{\delta})}{\delta^{1/N}}\right\},
\end{split}
\end{equation}
and we recall that 
\begin{equation}\label{eq:sopralivello}
\widetilde D_\delta=\{z\in B_{\kappa\delta^{-1/N}} : \widetilde m_\delta(z)=1\}
=\left\{z\in B_{\kappa\delta^{-1/N}} : w_{\delta}(z)>t_{\delta}\right\}.
\end{equation}

The core of the proof of Theorem \ref{thm:boundfrombelow} consists in bounding the optimal 
level $\IM$  of the limit problem \eqref{eq:lim_prob} in terms of a weighted Rayleigh quotient of $w_\delta$. An issue in this direction 
is that 
\[
|\widetilde D_\delta|=2 + o_\delta(1)\qquad\text{as }\delta\to0,
\]
but the error term cannot be discarded. This can be easily overcome using the reflection and 
scaling properties of the limit problem, namely
\[
\min\Big\{\lambda(A,\R^N):A\subset \R^N,\text{ measurable, }|A|=\ell\Big\} = \IM
\cdot \left(\frac{\ell}{2}\right)^{-2/N}, \qquad\text{for all }\ell>0. 
\]
\begin{lemma}
We have, as $\delta\to0$,
\begin{equation}\label{eq:IMwdelta}
\IM\cdot \left(\frac{|\widetilde D_\delta|}{2}\right)^{-2/N} \le 
 \lambda(\widetilde{D}_{\delta},\R^{N}) \leq \frac{\int_{B_{\kappa\delta^{-1/N}}}|\nabla w_\delta(z)|^2\,dz}{\int_{B_{\kappa\delta^{-1/N}}}\widetilde m_\delta w_\delta^2(z)\,dz}+o(\delta^{1/N}).
\end{equation}
\end{lemma}
\begin{proof}
We need to extend $w_\delta$ to the whole $\R^N$ to make it an admissible competitor for 
the limit problem. We extend $\widetilde m_\delta=-\beta$ in $\R^N\setminus B_{\kappa\delta^{-1/N}}$, as it is natural.
We define the new function\[
\widetilde w_\delta (z)=
\begin{cases}
w_\delta(z),\qquad &\text{ in }B_{\kappa \delta^{-1/N}},\\
h(z),\qquad &\text{ in }B_{(\kappa +1)\delta^{-1/N}}\setminus B_{\kappa \delta^{-1/N}},\\
0,\qquad &\text{ in }\R^N\setminus B_{(\kappa+1) \delta^{-1/N}},\\
\end{cases}
\]
where $h$  
 is the unique harmonic extension:\[
\begin{cases}
-\Delta h=0,\qquad &\text{ in }B_{(\kappa+1) \delta^{-1/N}}\setminus B_{\kappa \delta^{-1/N}},\\
h=w_\delta,\qquad &\text{ on }\partial B_{\kappa \delta^{-1/N}},\\
h=0,\qquad &\text{ on }\partial B_{(\kappa+1) \delta^{-1/N}}.\\
\end{cases}
\]
Thanks to Lemma~\ref{lem:expdecayw}, we have that \[
w_\delta(z)\leq C_3 e^{-C_4\kappa\delta^{-1/N}},\qquad \text{ on }\partial B_{\kappa \delta^{-1/N}},\\
\]
and the same also holds for the normal derivative, in view of \eqref{eq:decaygradw}.
As a consequence, calling for the sake of simplicity $A=B_{(\kappa +1)\delta^{-1/N}}\setminus B_{\kappa \delta^{-1/N}}$, we deduce
\[
\left|\int_A  -\beta \widetilde w_\delta^2(z)\,dz \right|\leq C |A|e^{-2C_4\kappa\delta^{-1/N}}\leq o(\delta^{1/N}).
\]
Moreover, using the Divergence Theorem, we obtain\[
\int_A|\nabla \widetilde w_\delta(z)|^2\,dz=\int_{\partial B_{\kappa\delta^{-1/N}}}\widetilde w_\delta \partial_\nu \widetilde w_\delta\,d\mathcal H^{N-1}\leq o(\delta^{1/N}).
\]
All in all, we have proved~\eqref{eq:IMwdelta}.
\end{proof}

In the next three lemmas, we proceed with the estimates of all the unknown terms in \eqref{eq:IMwdelta}, 
in terms of $\alpha_\delta$ and $u_\delta$. We start from the measure of $\widetilde D_\delta$.

\begin{lemma}\label{le:ImD}
We have, as $\delta\to0$,
\begin{equation}\label{eq:misuraDdeltatilde}
\left(\frac{|\widetilde D_\delta|}{2}\right)^{-2/N}=
1-\frac2N\alpha_\delta\delta^{1/N}\frac{\omega_{N-1}}{N+1}r_2^{N+1}+o(\delta^{1/N}).
\end{equation}
\end{lemma}

\begin{proof}
We recall that, thanks to Theorem~\ref{th:intro_D} (or to \cite[Lemma~4.5 and Proposition~4.7]{MPV3}), 
we have, as $\delta\to 0$
\begin{equation}\label{eq:inclusion}
(1-o_\delta(1))B_{r_{2}}\subset \widetilde D_\delta\subset (1+o_\delta(1))B_{r_{2}},
\end{equation}
recalling that $|B_{r_{2}}|=2$.
Then, we can compute
\[
\delta=|D_{\delta}|=\int_{\Omega}\ind{D_{\delta}}(x)\,dx= \int_{\R^{N}_+}\delta\ind{D_{\delta}}\Big(\Phi_{\delta}(\delta^{1/N}z)\Big)\,\det\Big(D\Phi_{\delta}(\delta^{1/N}z)\Big)\,dz,
\]
and using also~\cite[Lemma~A.1]{nitakagi_cpam} (see \eqref{eq:lemmaA1}), we obtain 
\[\begin{split}
2&=\int_{\R^N} \ind{\widetilde D_\delta}(z)\Big(1-\alpha_{\delta}{\delta}^{1/N}z_N+O(\delta^{2/N}|z|^2)\Big)\,dz
\\
&=|\widetilde D_\delta|-\alpha_\delta\delta^{1/N}\int_{\R^N}\ind{\widetilde D_\delta}(z)\Big(z_N+O(\delta^{1/N}|z|^{2})\Big)\,dz.
\end{split}\]
Using~\eqref{eq:inclusion}, it is easy to prove, as $\delta\to0$,
\[
2+\alpha_\delta\delta^{1/N}\int_{(1-o_\delta(1))B_{r_{2}}}z_N\,dz+O(\delta^{2/N})\leq |\widetilde D_\delta|\leq 2+\alpha_\delta\delta^{1/N}\int_{(1+o_\delta(1))B_{r_{2}}}z_N\,dz+O(\delta^{2/N}).
\]
We can now compute, by scaling, that \[
\int_{(1+o_\delta(1))B_{r_{2}}}z_N\,dz=(1+o_\delta(1))^{\frac{N+1}{N}}\int_{B_{r_{2}}}
y_N\,dy=(1+o_\delta(1))\frac{2\omega_{N-1}}{N+1}r_2^{N+1},
\]
where we made the change of variable $z=(1+o_\delta(1))^{1/N}y$ and we used
\[
\int_{(B_{r_{2}})^+}z_N\,dz=\frac{\omega_{N-1}}{N+1}r_2^{N+1},
\]
denoting as usual $r_2$ the radius of the ball of measure $2$.
As a consequence, we have, for $\delta\to0$, 
\begin{equation}\label{eq:misuraDdelta1}
|\widetilde D_\delta|= 2+2\alpha_\delta\delta^{1/N}\frac{\omega_{N-1}}{N+1}r_2^{N+1}+o(\delta^{1/N}).
\end{equation}
Thanks to these estimates, we obtain, as $\delta\to0$, 
\begin{equation*}
\frac{|\widetilde D_\delta|}{2}=
1+\alpha_\delta\delta^{1/N}\frac{\omega_{N-1}}{N+1}r_2^{N+1}+o(\delta^{1/N})
\end{equation*}
yielding the conclusion.
\end{proof}

Next we proceed with the expansion of the numerator of the Rayleigh quotients of $u_{\delta}$ and $w_{\delta}$. 
\begin{lemma}
Recalling that $\alpha_\delta = (N-1) H_{P_\delta}$ and that $\gamma$ is defined in \eqref{eq:Gamma}, 
we have, as $\delta\to0$,
\begin{equation}\label{eq:contogradboundbelow1}
\int_{\Omega}|\nabla u_\delta(x)|^2dx=\delta^{-2/N}\int_{B^+_{\kappa\delta^{-1/N}}}|\nabla 
w_{\delta} (z)|^2\,dz\left(1-\delta^{1/N}\frac{(N-1)\gamma\;\alpha_\delta}{\int_{\R^{N}_+}|\nabla w|^2}+o(\delta^{1/N})\right)
\end{equation}
\end{lemma}
\begin{proof}
Let $r>0$ such that $ B_r^+ \subset\Phi(B_\kappa^+)$.
From Corollary~\ref{cor:decayint}, we infer
\[
\left|\int_{\Omega}|\nabla u_\delta|^2-\int_{\Phi(B_\kappa^+)}|\nabla u_\delta|^2\right|\leq\int_{\Omega\setminus B_r^+}|\nabla u_\delta|^2\leq C_1 \delta^{-1 }e^{-C_2 r\delta^{-1/N}}\leq o(\delta^{1/N}).
\]
On the other hand, exploiting the usual change of variables $y=\Psi(x)$ and $z=y\delta^{-1/N}$, recalling \eqref{eq:wdelta} and applying~\cite[Lemma~A.1]{nitakagi_cpam}. 

\begin{equation}\label{eq:udeltawdelta}
\begin{split}
&\int_{\Phi(B_\kappa^+)}|\nabla u_\delta(x)|^2dx=\int_{B_\kappa^+}|\nabla v_\delta(y)|^2\,\left|\frac{y}{|y|}D\Psi(x)\right|^2\,\det D\Phi(y)\,dy\\
&=\delta^{-2/N}\int_{B^+_{\kappa\delta^{-1/N}}}|\nabla w_\delta(z)|^2\left[1+\delta^{1/N}z_N\Big(2\sum_{i,j=1}^{N-1}\psi_{ij}\frac{z_iz_j}{|z|^2}-\alpha_\delta\Big)+O(\delta^{2/N}|z|^2)\right]\,dz.
\end{split}
\end{equation}
At this point, one can check that there is a constant $C_0$, independent of $\delta$, such that \[
\left(1+\delta^{1/N}z_N\Big[2\sum_{i,j=1}^{N-1}\psi_{ij}\frac{z_iz_j}{|z|^2}-\alpha_\delta\Big]+O(\delta^{2/N}|z|^2)\right)\leq C_0,\qquad \text{for all }z\in B_{\kappa\delta^{-1/N}}^+.
\]
We now want to show that 
\begin{equation}\label{eq:opiccolo}
\int_{B^+_{\kappa\delta^{-1/N}}}\Big(|\nabla w_\delta|^2-|\nabla w|^2\Big)\left[\delta^{1/N}z_N\Big(2\sum_{i,j=1}^{N-1}\psi_{ij}\frac{z_iz_j}{|z|^2}-\alpha_\delta\Big)+O(\delta^{2/N}|z|^2)\right]\,dz=o(\delta^{1/N}),
\end{equation}
as $\delta\to 0.$
To prove this, we use the $H^1_{\loc}$ convergence of $w_\delta$ to $w$ (Proposition \ref{prop:vecchiolavoro}) and their exponential decay (Lemma~\ref{lem:expdecayw},  equation \eqref{eq:decay_w}).
Precisely, let us fix $\eps>0$ and find $R=R(\eps)>0$ such that, for all $\delta>0$ sufficiently small\[
\begin{split}
\int_{B^+_{\kappa\delta^{-1/N}}\setminus B_R^+} |\nabla w_\delta(z)|^2 
\left[z_N\Big(2\sum_{i,j=1}^{N-1}\psi_{ij}\frac{z_iz_j}{|z|^2}-\alpha_\delta\Big)+O(\delta^{1/N}|z|^2)\right]\,dz
&
\leq
C \int_{B^+_{\kappa\delta^{-1/N}}\setminus B_R^+} e^{-C_{4}|z|}|z| dz
\\
\leq C\left(
\delta^{-1}e^{-k\delta^{-1/N}}-R^{N}e^{-C_{4}R}\right)
&\leq 
 \eps,
\\
\int_{B^+_{\kappa\delta^{-1/N}}\setminus B_R^+} |\nabla w(z)|^2\left[z_N\Big(2\sum_{i,j=1}^{N-1}\psi_{ij}\frac{z_iz_j}{|z|^2}-\alpha_\delta\Big)+O(\delta^{1/N}|z|^2)\right]\,dz &\leq \eps.
\end{split}
\]
On the other hand,  $w_\delta$ converges strongly in $H^1(B_R)$ to $w$, hence 
\[
\int_{B^+_R}\Big(|\nabla w_\delta|^2-|\nabla w|^2\Big)\left[z_N\Big(2\sum_{i,j=1}^{N-1}\psi_{ij}\frac{z_iz_j}{|z|^2}-\alpha_\delta\Big)+O(\delta^{1/N}|z|^2)\right]\,dz\leq C(R)o_\delta(1)\leq \eps,
\]
up to take $\delta(\eps)$ small enough. All in all,   \eqref{eq:opiccolo} follows.

We also recall that, thanks to the exponential decay of $w$, it is clear that 
\begin{equation}\label{eq:gradwpiccolafuori}
\int_{(\R^{N}_+)\setminus B^+_{\kappa\delta^{-1/N}}}|\nabla w|^2\,dz=o(\delta^{1/N}),\qquad \text{as }\delta\to 0.
\end{equation}

We can now manage the higher order terms in~\eqref{eq:udeltawdelta}, using~\eqref{eq:opiccolo} and~\eqref{eq:gradwpiccolafuori}
\[
\begin{split}
&\int_{B^+_{\kappa\delta^{-1/N}}}|\nabla w_\delta(z)|^2\left[\delta^{1/N}z_N\Big(2\sum_{i,j=1}^{N-1}\psi_{ij}\frac{z_iz_j}{|z|^2}-\alpha_\delta\Big)+O(\delta^{2/N}|z|^2)\right]\,dz\\
&=\int_{B^+_{\kappa\delta^{-1/N}}}\Big(|\nabla w_\delta(z)|^2-|\nabla w(z)|^2+|\nabla w(z)|^2\Big)\left[\delta^{1/N}z_N\Big(2\sum_{i,j=1}^{N-1}\psi_{ij}\frac{z_iz_j}{|z|^2}-\alpha_\delta\Big)+O(\delta^{2/N}|z|^2)\right]\,dz\\
&=\int_{B^+_{\kappa\delta^{-1/N}}}|\nabla w(z)|^2\left[\delta^{1/N}z_N\Big(2\sum_{i,j=1}^{N-1}\psi_{ij}\frac{z_iz_j}{|z|^2}-\alpha_\delta\Big)+O(\delta^{2/N}|z|^2)\right]\,dz+o(\delta^{1/N})\\
&=\int_{\R^{N}_+}|\nabla w(z)|^2\left[\delta^{1/N}z_N\Big(2\sum_{i,j=1}^{N-1}\psi_{ij}\frac{z_iz_j}{|z|^2}-\alpha_\delta\Big)+O(\delta^{2/N}|z|^2)\right]\,dz+o(\delta^{1/N})\\
&=-\delta^{1/N}[(N-1)\alpha_\delta\gamma]+o(\delta^{1/N}).
\end{split}
\] 
We note that the last equality follows with computations similar to the ones of~\cite[Step 1 of the proof of Proposition~3.5]{MPV3}, in particular using that\[
\sum_{i,j=1}^{N-1}\psi_{ij}\int_{\R^{N}_+}|\nabla w(z)|^2\frac{z_iz_j}{|z|^2}z_N\,dz=\gamma\;\alpha_\delta.
\]
Then, coming back to \eqref{eq:udeltawdelta}, we have  
\begin{equation}\label{eq:contogradboundbelow}
\begin{split}
\int_{\Omega}|\nabla u_\delta(x)|^2&=
\delta^{-2/N}\int_{B^+_{\kappa\delta^{-1/N}}}|\nabla w_\delta(z)|^2\,dz
\left(1-\delta^{1/N}
\frac{(N-1)\gamma\;\alpha_\delta}{\int_{B^+_{\kappa\delta^{-1/N}}}|\nabla w_\delta(z)|^2\,dz}+o(\delta^{1/N})\right).
\end{split}
\end{equation}
Moreover, using again the convergence and the exponential decay of $w_\delta$ and $w$, it is clear that \[
\int_{\R^{N}_+}|\nabla w|^2\,dz=\int_{B_{\kappa\delta^{-1/N}}^+}|\nabla w_\delta |^2\,dz+o_\delta(1),\qquad\text{as }\delta\to 0,
\]
hence,  ~\eqref{eq:contogradboundbelow} yields the conclusion.
\end{proof}
We use a similar strategy for the denominator of the Rayleigh quotient, with the key tools being again Proposition~\ref{prop:vecchiolavoro} and Lemma~\ref{lem:expdecayw}. 

\begin{lemma}
It results
\begin{equation}\label{eq:contodenom}
\begin{split}
\int_{\Omega}m_\delta(x) u_\delta(x)^2\,dx &=\int_
{B_{\kappa\delta^{-1/N}}}\widetilde m_\delta w_\delta^2\,dz\Big(1-\delta^{1/N}\frac{\gamma_1\;\alpha_\delta}{\int_{\R^{N}_+}mw^2}+o(\delta^{1/N})\Big),
\\
\text{where }   \gamma_1 &=\int_{\R^{N}_+}m(z)w(z)^2z_N\,dz.
\end{split}
\end{equation}
\end{lemma}
\begin{proof}
First, as $\Phi(B_\kappa^+)\supset B_r^+$ for some $r>0$,   
the exponential decay of $u_\delta$ (see Corollary~\ref{cor:decayint}) and the fact that $-\beta\leq m_\delta\leq 1$, imply
\[
\left|\int_{\Omega}m_\delta u_\delta^2-\int_{\Phi(B_\kappa^+)}m_\delta u_\delta^2\right|\leq \max{\{\beta,1\}}\int_{\Omega\setminus B_R^+} u_\delta^2\leq C_1 e^{-C_2 R\delta^{-1/N}}\leq o(\delta^{1/N}), \;\text{  as $\delta\to 0$.}
\]
Then we have,
\[
\begin{split}
&\int_{\Omega}m_\delta(x) u^2_\delta(x)\,dx=\int_{B_\kappa^+}\widetilde m_\delta(y\delta^{-1/N})v_\delta^2(y)\det D\Phi(y)\,dy\\
&=\int_{B^+_{\kappa\delta^{-1/N}}}\widetilde m_\delta(z)w^2_\delta(z)\det D\Phi(y)\,dz\\
&=\int_{B^+_{\kappa\delta^{-1/N}}}\widetilde m_\delta(z)w^2_\delta(z)\Big(1-\alpha_\delta z_N\delta^{1/N}+O(\delta^{2/N}|z|^2)\Big)\,dz.
\end{split}
\]
As above, we have that there is a constant $C_0$, independent of $\delta$, such that \[
\left|1-\alpha_\delta\delta^{1/N}z_N+O(\delta^{2/N}|z|^2)\right|\leq C_0,\qquad \text{for all }z\in B_{\kappa\delta^{-1/N}}^+.
\]

With the same strategy as the one used for the gradient, we first prove that, as $\delta\to0$,
\begin{equation}\label{eq:erroremw}
\int_{B^+_{\kappa\delta^{-1/N}}}|\widetilde m_\delta w_\delta^2- m w^2|\Big(\alpha_\delta z_N\delta^{1/N}+O(\delta^{2/N}|z|^2)\Big)\,dz=o(\delta^{1/N}),
\end{equation}
which we split in two easier steps,
\begin{equation}\label{eq:erroremwsteps}
\begin{split}
\int_{B^+_{\kappa\delta^{-1/N}}}|\widetilde m_\delta w_\delta^2-\widetilde m_\delta w^2|\Big(\alpha_\delta z_N\delta^{1/N}+O(\delta^{2/N}|z|^2)\Big)\,dz=o(\delta^{1/N}),
\\
\int_{B^+_{\kappa\delta^{-1/N}}}|\widetilde m_\delta w^2- m  w^2|\Big(\alpha_\delta z_N\delta^{1/N}+O(\delta^{2/N}|z|^2)\Big)\,dz=o(\delta^{1/N}).
\end{split}
\end{equation}

To prove the estimates in~\eqref{eq:erroremwsteps}, we use the $H^1_{\loc}$ convergence of $w_\delta$ to $w$ and their exponential decay. We start from the first.
More precisely, let us fix $\eps>0$ and find $R=R(\eps)>0$ such that, for all $\delta>0$ sufficiently small\[
\begin{split}
\int_{B^+_{\kappa\delta^{-1/N}}\setminus B_R^+} w_\delta(z)^2\left[\alpha_\delta z_N+O(\delta^{1/N}|z|^2)\right]\,dz\leq \eps,\\
\int_{B^+_{\kappa\delta^{-1/N}}\setminus B_R^+}  w(z)^2\left[\alpha_\delta z_N+O(\delta^{1/N}|z|^2)\right]\,dz\leq \eps,
\end{split}
\]
On the other hand, in $B_R$, $w_\delta$ converges strongly in $H^1$ to $w$ (and $-\beta\leq \widetilde m_\delta\leq 1$), hence 
\[
\int_{B^+_R}\widetilde m_\delta ( w_\delta^2- w^2)\left[\alpha_\delta z_N+O(\delta^{1/N}|z|^2)\right]\,dz\leq C(R)o_\delta(1)\leq \eps,
\]
so that the first estimate in \eqref{eq:erroremwsteps} follows.

The second estimate in~\eqref{eq:erroremwsteps} can be proved in a similar way.
We fix again $\eps>0$ and find $R=R(\eps)>0$ such that, for all $\delta>0$ sufficiently small\[
\begin{split}
\int_{B^+_{\kappa\delta^{-1/N}}\setminus B_R^+}  (\widetilde m_\delta- m)w(z)^2\left[\alpha_\delta z_N+O(\delta^{1/N}|z|^2)\right]\,dz\leq \eps,
\end{split}
\]
which is possible thanks to the exponential decay of $w$ and the boundedness of $\widetilde m_\delta$ and $m$.
In $B_R$, $\widetilde m_\delta$ converges weakly $*$ in $L^\infty$ to $m$ (and the other terms in the integral are clearly $L^1$), hence 
\[
\int_{B^+_R}(\widetilde m_\delta- m) w^2\left[\alpha_\delta z_N+O(\delta^{1/N}|z|^2)\right]\,dz\leq C(R)o_\delta(1)\leq \eps,
\]
up to take $\delta(\eps)$ small enough.
In conclusion we have proved~\eqref{eq:erroremw}.

We also recall that, thanks to the exponential decay of $w$ (and the fact that $-\beta\leq m\leq 1$), it is clear that 
\begin{equation}\label{eq:wpiccolafuori}
\int_{\R^{N}_+\setminus B^+_{\kappa\delta^{-1/N}}}m w^2\,dz=o(\delta^{1/N}),\qquad \text{as }\delta\to 0.
\end{equation}

From~\eqref{eq:erroremw}, adding and subtracting the suitable terms, we obtain \[
\begin{split}
&\int_{B^+_{\kappa\delta^{-1/N}}}\widetilde m_\delta(z)w_\delta(z)^2\Big(-\alpha_\delta z_N\delta^{1/N}+O(\delta^{2/N}|z|^2)\Big)\,dz\\
&=\int_{B^+_{\kappa\delta^{-1/N}}}(\widetilde m_\delta w_\delta^2-mw^2)+ m(z)w(z)^2\Big(-\alpha_\delta z_N\delta^{1/N}+O(\delta^{2/N}|z|^2)\Big)\,dz\\
&=\int_{B^+_{\kappa\delta^{-1/N}}} m(z)w(z)^2\Big(-\alpha_\delta z_N\delta^{1/N}+O(\delta^{2/N}|z|^2)\Big)\,dz+o(\delta^{1/N})=-\alpha_\delta\delta^{1/N}\gamma_1+o(\delta^{1/N}).
\end{split}
\]
All in all,  we have 
\[
\int_{B^+_{\kappa\delta^{-1/N}}}\widetilde m_\delta w_\delta^2=\int_{\R^{N}_+}mw^2+o_\delta(1)\,,
\]
showing the conclusion.
\end{proof}

\begin{proof}[Proof of Theorem \ref{thm:boundfrombelow}]
Putting together~\eqref{eq:contogradboundbelow1} and~\eqref{eq:contodenom}, we have 
\begin{equation}\label{eq:boundbelowwdelta}
\od(\delta)= \delta^{-2/N}\frac{\int_{B_{\kappa\delta^{-1/N}}}|\nabla w_\delta|^2}{\int_{B_{\kappa\delta^{-1/N}}}\widetilde m_\delta w_\delta^2}\,\Big(1-\delta^{1/N}\frac{(N-1)\gamma
\;\alpha_\delta}{\int_{\R^{N}_+}|\nabla w|^2}+o(\delta^{1/N})\Big)\Big(1+\delta^{1/N}\frac{
\gamma_1\;\alpha_\delta}{\int_{\R^{N}_+}mw^2}+o(\delta^{1/N})\Big),
\end{equation}
which, in view of \eqref{eq:IMwdelta},   implies
\begin{equation}\label{eq:boundbassoparziale}
\od(\delta)\geq
\delta^{-2/N}\Big(1-\frac{\delta^{1/N}\;\alpha_\delta}{\IM\int_{\R^{N}_+}mw^2}\Big[(N-1)\gamma-\IM\gamma_1\Big]+o(\delta^{1/N})\Big) \cdot \IM \left(\frac{|\widetilde D_\delta|}{2}\right)^{-2/N}
\end{equation}
In turn, \eqref{eq:misuraDdeltatilde} yields
\[
\begin{split}
\od(\delta)\geq \delta^{-2/N}\IM\left(1-\frac{\delta^{1/N}\;\alpha_\delta}{\IM\int_{\R^{N}_+}mw^2}[(N-1)\gamma-\IM\gamma_1]+o(\delta^{1/N})\right)\times
\\
\times\left(1-2/N\delta^{1/N}\;\alpha_\delta\frac{\omega_{N-1}}{N+1}r_2^{N+1}+o(\delta^{1/N})\right).
\end{split}\]
As
\begin{equation}\label{eq:uguaglianzamagica}
(N-1)\gamma-\IM\gamma_1=2\gamma-2\IM r_2^{N+1}\frac{\omega_{N-1}}{N(N+1)}\int_{\R^{N}_+}mw^2,
\end{equation}
the theorem follows.
\end{proof}

\begin{remark}\label{rmk:scorciatoia}
With a closer look at the previous proof, and in particular at the role of \eqref{eq:IMwdelta} in estimate \eqref{eq:boundbassoparziale}, 
we notice that also the following inequality holds true:
\begin{equation*}
\begin{split}
\od(\delta)&\geq 
\delta^{-2/N}\left(1-\Gamma\widehat H\;\delta^{1/N}+o(\delta^{1/N})\right)\cdot \left(\frac{|\widetilde D_\delta|}{2}\right)^{2/N}\lambda(\widetilde D_\delta,\R^N),\\
&\geq 
\delta^{-2/N}\left(1-\Gamma\widehat H\;\delta^{1/N}+o(\delta^{1/N})\right)\cdot \IM
\end{split}
\end{equation*}
where we used also \eqref{eq:Gamma1} and Theorem \ref{thm:sviluppoesatto}. As a consequence,
\[
\delta^{2/N}\od(\delta) = \left(\frac{|\widetilde D_\delta|}{2}\right)^{2/N}\lambda(\widetilde D_\delta,\R^N) + o_\delta(1) = \IM + o_\delta(1)
\] 
as $\delta\to0$, and finally
\begin{equation*}
\delta^{2/N}\od(\delta)\geq 
\left(\frac{|\widetilde D_\delta|}{2}\right)^{2/N}\lambda(\widetilde D_\delta,\R^N) 
-\Gamma\widehat H\;\IM\;\delta^{1/N}+o(\delta^{1/N}),
\end{equation*}
which establishes a lower bound for $\Lambda(\delta)$ on terms of an eigenvalue of the set $\widetilde D_\delta$ related to the limit problem. 
\end{remark}

\section{Polar parameterization of the optimal sets}\label{sec:polar}

In the last part of the paper we are going to exploit the blow up analysis performed in the previous 
section, and in particular the asymptotic radial symmetry of the rescaled eigenfunctions $w_\delta$ and 
optimal sets $\widetilde{D}_{\delta}$, see \eqref{eq:inclusion}, to obtain a polar parametrization 
of $\partial\widetilde{D}_{\delta}$ and to investigate its finer regularity properties, also from a 
quantitative point of view. Of course, such information can be translated to $D_\delta$ using the 
diffeomorphisms introduced at the beginning of Section \ref{sec:boundbelow}.

To this aim, let us first recall the concept of nearly spherical set (used in the proof of a quantitative isoperimetric inequality first in~\cite{fuglede}). For our aims, although 
$\widetilde{D}_{\delta}$ has measure $2$ only in the limit, it is convenient 
to normalize the reference radius to $r_2$, where as usual $|B_{r_{2}}|=2$. 
\begin{definition}\label{def:nearlysph}
A (bounded) set $A\subset \R^{N}$ is a nearly spherical set of class $C^{k,\alpha}$, 
centered at $Q$, if  there exists 
$\varphi=\varphi_A\in C^{k,\alpha}(\mathbb{S}^{N-1})$ with $\|\varphi\|_{L^{\infty}}\leq r_2/2$ such that 
\[
\partial A=\left\{x\in \R^{N} : x=Q+(r_{2}+\varphi(\theta))\theta, \text{for $\theta \in \mathbb{S}^{N-1}$}\right\}.
\]
In such case, we say that $A$ is parametrized by $\varphi$.
\end{definition}

As a matter of fact, it is not difficult to exploit \eqref{eq:inclusion} and the implicit function 
theorem to see that $\widetilde{D}_{\delta}$ is nearly spherical, centered at the maximum point 
$0 = \Psi_\delta(P_\delta)$, at least when $\delta$ is sufficiently small. On the other hand, in 
order to employ suitable quantitative estimates obtained in \cite{ferreri_verzini2}, 
it is necessary to adjust the center of the parametrization, choosing instead the barycenter 
$\bari(\widetilde{D}_{\delta})$ (and the corresponding preimage on $\partial\Omega$). This is possible 
thanks to the following lemma, where we keep track of $\Psi_\delta(P_\delta)=0$ for the sake of clarity.

\begin{lemma}\label{lem:bari}
Under the notation of Section \ref{sec:boundbelow} we have, as $\delta\to0$:
\begin{enumerate}
\item $\bari(\widetilde{D}_{\delta}) \in \{z:z_N=0\}$ and $\left|\bari(\widetilde{D}_{\delta})
- \Psi_\delta(P_\delta)\right|=o_\delta(1)$;
\item $Q_{\delta}:=\Phi_\delta(\bari(\widetilde{D}_{\delta})) \in \partial D_{\delta}$ and 
$|Q_{\delta}-P_{\delta}|=o(\delta^{1/N})$.
\end{enumerate}
\end{lemma}
\begin{proof}
First, by symmetry of $\widetilde{D}_{\delta}$, we obtain that $\bari(\widetilde{D}_{\delta}) \in \{z:z_N=0\}$ and thus $Q_{\delta}\in \partial D_{\delta}$ (recall equation \eqref{eq:Dtilde}). 

Next, for every $\eps>0$,
let $B_{r_{-}}(\Psi_\delta(P_{\delta}))$, and $B_{r_{+}}(\Psi_\delta(P_{\delta}))$  be such that 
$|B_{r_{\pm}}|=2(1\pm\eps)$ and, for $\delta$ sufficiently small,  
\[
 B_{r_{-}}\subset \widetilde D_\delta\subset B_{r_{+}}.
\]
We infer
\[
\begin{split}
\left| \bari(\widetilde{D}_{\delta})-\Psi_\delta(P_{\delta})\right|
&\leq
\frac1{|\widetilde{D}_{\delta}|}\left|\int_{\widetilde{D}_{\delta}}\left(x-\Psi_\delta(P_{\delta})\right)dx\right|
=\frac1{|\widetilde{D}_{\delta}|}\left|\int_{\widetilde{D}_{\delta}\setminus B_{r_{-}}(\Psi_\delta(P_{\delta}))}\left(x-\Psi_\delta(P_{\delta})\right)dx\right|
\\
&\leq  \frac1{|\widetilde{D}_{\delta}|}\int_{B_{r_{+}}(\Psi_\delta(P_{\delta}))\setminus B_{r_{-}}(\Psi_\delta(P_{\delta}))}|x-\Psi_\delta(P_{\delta})|dx
\leq C \eps (1+o_\delta(1))
\end{split}
\]
and the lemma follows.
\end{proof}

Taking into account the lemma above, we change reference system in the blow-up analysis by a vanishing translation, and from now on we assume that
\begin{equation}\label{eq:cambiotrasla}
\Psi_\delta(Q_{\delta})=\bari(\widetilde{D}_{\delta}) = 0,\qquad \Psi_\delta(P_{\delta}) = -\bari(\widetilde{D}_{\delta}).
\end{equation}
Although in principle this changes the definitions of  $\widetilde{D}_{\delta}$, $w_\delta$, 
$\Psi_\delta$ and so on, Lemma \ref{lem:bari} implies that all the results in Section 
\ref{sec:boundbelow} hold true also in the new reference system; in particular, equation 
\eqref{eq:inclusion} holds with balls centered at $0=\bari(\widetilde{D}_{\delta})$.


\begin{proposition}\label{pro:nearlysphericalalpha}
For $\delta$ sufficiently small, $\widetilde{D}_{\delta}$ is nearly spherical of class 
$C^{1,\alpha}$, $0<\alpha<1$, centered at $0=\bari(\widetilde{D}_{\delta})$ and  parametrized by 
$\varphi_{\delta}$. In addition
\[
\|\varphi_{\delta}\|_{C^{1,\alpha}(\mathbb{S}^{N-1})}\to 0, \qquad \text{as $\delta\to 0$}.
\]
\end{proposition}
\begin{proof}
The proof can be obtained (with obvious changes ) as in \cite[Proposition 3.11]{ferreri_verzini2}.
Indeed, using polar coordinates and recalling \eqref{eq:sopralivello} we can write
\[
\widetilde{D}_{\delta}=\left\{z\in B_{k\delta^{-1/N} } : w_{\delta}(z)>t_{\delta}\right\}=\left\{w_{\delta}(\rho\theta)>t_{\delta}\right\}
\]
where $\rho>0$ and $\theta=\frac{x}{|x|} \in {\mathbb S}^{N-1}$. Moreover, in view of  \eqref{eq:inclusion}  one has
\[
\partial\widetilde{D}_{\delta}\subset (1+o_{\delta}(1))B_{r_2}\setminus (1-o_{\delta}(1))B_{r_2}\subset B_{\overline{r}}\setminus B_{\underline{r}},
\]
for some $0<\underline{r}<r_{2}<\overline{r}$. 
Let us consider $F(\vfi,\theta):=w_{\delta}((r_{2}+\vfi)\theta)$ for $\vfi=\rho-r_{2}$
and $\rho \in [\underline{r},\overline{r}]$.
As $w_{\delta}$ converges to $w$ in $C^{1,\alpha}(\overline{B}_{\overline{r}})$, so that
\[
\max_{\overline{B_{\overline{r}}\setminus B_{\underline{r}}}}\partial_{r} w_{\delta}<\frac12\max_{\overline{B_{\overline{r}}\setminus B_{\underline{r}}}}\partial_{r} w <0,
\quad \text{for every $\theta \in {\mathbb S}^{N-1}$. }
\]
Then we can apply the Implicit Function Theorem  to the function $F(\rho,\theta)$, obtaining 
$F(\rho,\theta)=t_{\delta}$ if and only if 
$\rho=\rho(\theta)$. This argument can be implemented for every $\theta \in 
{\mathcal \mathbb{S}^{N-1}}$ , so that by compactness we obtain a globally defined 
$\varphi_{\delta}(\theta)=\rho(\theta)-r_{2}$. 
Since  
$w_{\delta} $ is of class $C^{1,\alpha}$,  the same regularity holds for  
$\varphi_{\delta}$.
Furthermore, 
\[
\nabla \varphi_{\delta}=-\frac{\left(r_{2}+\varphi_{\delta}(\theta)\right)}{\partial_{\rho}F(\rho,\theta)}\nabla_{T}w_{\delta}
=-\frac{\left(r_{2}+\varphi_{\delta}(\theta)\right) }{\partial_{\rho}F(\rho,\theta)}\left(
\nabla w_{\delta}-(\nabla w_{\delta} \cdot \theta)\theta\right).
\]
where $\nabla_{T}w_{\delta}$ denotes the tangential component of the gradient of $w_{\delta}$, and  $\theta=\frac{x}{|x|}$
Then the conclusion is a direct consequence of the $C^{1,\alpha}$ convergence of $w_{\delta}$.
\end{proof}

In order to improve the asymptotic information provided in the last proposition, from now on 
we assume that, for some $\theta>0$, 
\[
\partial \Omega \text{ is of class }C^{3,\theta},
\]
so that $\Phi_\delta$ is of class $C^{2,\theta}$; we are going to prove a $C^{1, 1}$ 
global regularity result (up to the boundary) for the eigenfunctions $w_\delta$. From this, 
we will deduce the $C^{1, 1}$ regularity for the nearly spherical representation $\varphi_\delta$. 
This yields a $C^{1, 1}$ convergence result of $\widetilde{D}_{\delta}$ to the ball with 
measure two, thus improving Proposition \ref{pro:nearlysphericalalpha}. 

To this aim, we follow the strategy of \cite[Section 5]{ferreri_verzini2}, where 
analogous results were derived for a similar problem with Dirichlet boundary conditions. 
In particular, we 
deduce the desired regularity as a corollary of the regularity for transmission problems 
(see \cite{Caffarelli2021:TransmissionProblems, Dong2021:TransmissionProblems}). 
The main difficulty in extending the regularity results in \cite[Section 5]{ferreri_verzini2} 
to the problem we consider here is the positioning of the optimal favorable regions: in the case of 
homogeneous Dirichlet boundary conditions the favorable regions asymptotically 
concentrate in the interior of $\Omega$ (\cite[Theorem 1.1]{ferreri_verzini2}), while for 
Neumann boundary conditions the concentration occurs at $\partial \Omega$ (Theorem 
\ref{th:intro_D}). Hence, \cite[Section 5]{ferreri_verzini2} essentially deals with interior 
regularity results, and we need to extend all results up to the boundary of $\Omega$.
\begin{proposition}
\label{prop:TransmissionPb}
  Assume that $\partial \Omega$ is of class $C^{3,\theta}$. Then, for all $i = 1, \dots, N-1$, the functions $\partial_i w_{\delta}$ are $H^1\left( B_{\kappa\delta^{-1/N}}^+ \right)$-solutions of the transmission problem
\begin{equation}\label{eqn:TransmissionPb}
\begin{cases}
\begin{aligned}
-\diverg( A^{\delta} \nabla &\partial_i w_{\delta} ) 
 = \diverg( (\partial_i A^{\delta} )\nabla w_{\delta} ) \\ &+ \delta^{2/N} \od(\delta) \widetilde m_\delta (J_{\delta} \partial_i w_{\delta} + w_{\delta} \partial_i J_{\delta}  )
\end{aligned} 
  & \text{in } \left( \widetilde D_\delta \cap B_{\kappa\delta^{-1/N}}^+ \right) \bigcup \left( B_{\kappa\delta^{-1/N}}^+ \setminus \overline{\widetilde D_\delta} \right) , \\
[\partial_i w_{\delta}] = 0,& \text{on } \partial\widetilde D_\delta \cap B_{\kappa\delta^{-1/N}}^+, \\
[A^{\delta} \nabla ( \partial_i w_{\delta}) \cdot \nu] = - \delta^{2/N} \od(\delta) (1+\beta) w_{\delta} J_{\delta} \nu_i & \text{on } \partial\widetilde D_\delta \cap B_{\kappa\delta^{-1/N}}^+, \\
\partial_N \partial_i w_{\delta} = 0 = A^{\delta} \nabla \partial_i w_{\delta} \cdot e_N & \text{on }
 \{z_N = 0 \} \cap B_{\kappa\delta^{-1/N}} ,
\end{cases}
\end{equation}
where $[\cdot]$ denotes the jump across $\partial\widetilde D_\delta \cap B_{\kappa\delta^{-1/N}}^+$, 
$\nu_i$ denotes the $i$-th component of the outer unit normal to $\widetilde D_\delta \cap B_{\kappa\delta^{-1/N}}^+$, and $e_N$ denotes the normal to the hyperplane $\{z_N = 0 \}$. Moreover, there exists a constant $C>0$ such that
\begin{equation}\label{eqn:C2aUniformBound}
\| w_{\delta} \|_{C^{2, \theta} \left( \overline{\widetilde D_\delta \cap B_{\kappa\delta^{-1/N}}^+} \right)} + \| w_{\delta} \|_{C^{2, \theta} \left( \overline{ B_{\kappa\delta^{-1/N}}^+ \setminus \widetilde D_\delta } \right)} \le C,
\end{equation}
uniformly as $\delta \to 0$.
\end{proposition}
\begin{proof}
Let us start with \eqref{eqn:TransmissionPb}. To begin with, notice that $\partial_N \partial_i w_{\delta} = 0 = A^{\delta} \nabla \partial_i w_{\delta} \cdot e_N$ a.e. on $\{z_N = 0 \} \cap \overline{B_{\kappa\delta^{-1/N}}^+}$, since we are considering only derivatives of $w_{\delta}$ tangential to $\{z_N = 0 \}$. Using the integration by parts formula, for any $\varphi \in C^{\infty}_c\left( B_{\kappa\delta^{-1/N}}^+ \right)$ we have
\begin{align}\label{eqn:eqn1C2aProp}
\begin{split}
    & \delta^{2/N} \od(\delta) \int_{B_{\kappa\delta^{-1/N}}^+} \widetilde m_\delta \varphi J_{\delta} \partial_i w_{\delta} = 
    \delta^{2/N} \od(\delta) \left[ \int_{\widetilde D_\delta \cap B_{\kappa\delta^{-1/N}}^+} \varphi J_{\delta} \partial_i w_{\delta} - 
    \beta \int_{B_{\kappa\delta^{-1/N}}^+ \setminus \widetilde D_\delta} \varphi J_{\delta} \partial_i w_{\delta} \right] = \\
    & =\delta^{2/N} \od(\delta) (1+\beta) \int_{\partial \widetilde D_\delta \cap B_{\kappa\delta^{-1/N}}^+} w_{\delta} J_{\delta} \varphi \nu_i - \delta^{2/N} \od(\delta) \int_{B_{\kappa\delta^{-1/N}}^+} \widetilde m_\delta \left( w_{\delta} J_{\delta} \partial_i \varphi + w_{\delta} \varphi \partial_i J_{\delta} \right),
    \end{split}
\end{align}
Now, using \eqref{eq:eq_in_forma_di_div} and integrating again by parts we have that
\begin{align}\label{eqn:eqn2C2aProp}
\begin{split}
    & \delta^{2/N} \od(\delta) \int_{B_{\kappa\delta^{-1/N}}^+} \widetilde m_\delta w_{\delta} J_{\delta}\partial_i \varphi = \delta^{2/N} \od(\delta) \left[ \int_{\widetilde D_\delta \cap B_{\kappa\delta^{-1/N}}^+} w_{\delta} J_{\delta}\partial_i \varphi  - 
    \beta \int_{B_{\kappa\delta^{-1/N}}^+ \setminus \widetilde D_\delta} w_{\delta}J_{\delta} \partial_i \varphi \right] = 
    \\
    & = - \left[ \int_{\widetilde D_\delta \cap B_{\kappa\delta^{-1/N}}^+} \diverg\left( A^{\delta} \nabla w_{\delta} \right) \partial_i \varphi  +
    \int_{B_{\kappa\delta^{-1/N}}^+ \setminus \widetilde D_\delta} \diverg\left( A^{\delta} \nabla w_{\delta} \right) \partial_i \varphi \right] = 
    \\
    & = \int_{B_{\kappa\delta^{-1/N}}^+} A^{\delta} \nabla w_{\delta} \cdot  \partial_i \nabla \varphi = - \int_{B_{\kappa\delta^{-1/N}}^+} A^{\delta} \nabla \partial_i w_{\delta} \cdot \nabla \varphi - \int_{B_{\kappa\delta^{-1/N}}^+} \partial_i \left(A^{\delta}\right) \nabla w_{\delta} \cdot \nabla \varphi.
\end{split}
\end{align}
Combining \eqref{eqn:eqn1C2aProp} and \eqref{eqn:eqn2C2aProp} we obtain \eqref{eqn:TransmissionPb}.

Now we turn to \eqref{eqn:C2aUniformBound}. Notice that the condition $\partial_N \partial_i w_{\delta} = 0$ on $\{z_N = 0 \} \cap \overline{B_{\kappa\delta^{-1/N}}^+}$ and the properties of the diffeomorphism $\Phi_{\delta}$ (see Section \ref{sec:boundbelow}) allow to extend $\partial_i w_{\delta}$ by reflection on the whole $B_{\kappa\delta^{-1/N}}$, together with its equation. Again, this reflected function solves a transmission problem, but this time on the whole $B_{\kappa\delta^{-1/N}}$. More precisely,
\begin{equation}\label{eqn:TransmissionPbReflected}
\begin{cases}
-\diverg\left( A^{\delta} \nabla \partial_i w_{\delta} \right) = \delta^{2/N} \od(\delta) \widetilde m_\delta \left(J_{\delta} \partial_i w_{\delta} + w_{\delta} \partial_i J_{\delta}  \right) + \diverg\left( G_i^{\delta} \right) & \text{in } \widetilde D_\delta \bigcup \left( B_{\kappa\delta^{-1/N}} \setminus \overline{\widetilde D_\delta} \right) , \\
[\partial_i w_{\delta}] = 0, \quad [A^{\delta} \nabla ( \partial_i w_{\delta}) \cdot \nu] = - \delta^{2/N} \od(\delta) (1+\beta) w_{\delta}J_{\delta} \nu_i 
& \text{on } \partial\widetilde D_\delta ,
\end{cases}
\end{equation}
where $G_i^{\delta}$ denotes the even extension of the function $\left(\partial_i A^{\delta} \right) \nabla w_{\delta}$ with respect to $\{ z_N = 0 \}$. Since $i\neq N$ and recalling Remark \ref{rmk:non_div_form}, we have that this function is of class $C^{0, \theta}$. Hence, as a consequence of \cite[Theorem 1.2]{Dong2021:TransmissionProblems} we have, for all $i = 1, \dots, N-1$
\begin{equation}\label{eqn:eqn3C2aProp}
\| \partial_i w_{\delta} \|_{C^{1, \theta} \left( \overline{\widetilde D_\delta \cap B_{\kappa\delta^{-1/N}}^+} \right)} + \| \partial_i w_{\delta} \|_{C^{1, \theta} \left( \overline{ B_{\kappa\delta^{-1/N}}^+ \setminus \widetilde D_\delta } \right)} \le C.
\end{equation}
Strictly speaking, the application of \cite[Theorem 1.2]{Dong2021:TransmissionProblems} to 
\eqref{eqn:TransmissionPbReflected} requires 
two adjustments: indeed, such result applies to solutions to transmission problems having fixed 
interface, while here $\partial\widetilde D_\delta$ depends on $\delta$, and homogeneous Dirichlet 
boundary conditions. Actually, by small perturbations, we can modify \eqref{eqn:TransmissionPbReflected} 
to meet both these conditions. We describe such arguments in full details in the proof of Proposition 
\ref{prop:TangentDerivaticesC11To0} ahead, where the same issues have to be faced for a related  transmission problem, see \eqref{eqn:TransmissionPbvdeltaReflected} and \eqref{eqn:TransmissionPbDifference}, to obtain a more delicate 
estimate.

We are left to prove that
\begin{equation}\label{eqn:eqn4C2aProp}
\left\| \partial^2_{N} w_{\delta} \right\|_{C^{0, \theta} \left( \overline{\widetilde D_\delta \cap B_{\kappa\delta^{-1/N}}^+} \right)} + \left\| \partial^2_{N} w_{\delta} \right\|_{C^{0, \theta} \left( \overline{ B_{\kappa\delta^{-1/N}}^+ \setminus \widetilde D_\delta } \right)} \le C.
\end{equation}
Condition \eqref{eqn:eqn4C2aProp} follows isolating $\partial^2_{N} w_{\delta}$ in the equation \eqref{eq:eq_in_forma_di_div} restricted to $B_{\kappa\delta^{-1/N}}^+$, and using \eqref{eqn:eqn3C2aProp}. Hence, the proof is concluded.
\end{proof}
A direct consequence of Proposition \ref{prop:TransmissionPb} is the following
\begin{corollary}
    There exists a constant $C>0$ such that
    \[
    \| w_{\delta} \|_{C^{1, 1} \left( \overline{B_{\kappa\delta^{-1/N}} } \right)} \le C,
    \]
    uniformly as $\delta \to 0$. 
\end{corollary}
\begin{proof}
    It is sufficient to notice that Proposition \ref{prop:TransmissionPb} gives the $C^{1, 1}$ regularity of $w_{\delta}$ on the closed half ball $\overline{B_{\kappa\delta^{-1/N}}^+}$, and that, since $\partial_N w_{\delta} = 0$ on $\{ z_N = 0 \}$, such regularity is preserved by reflection with respect to $\{ z_N = 0 \}$. Moreover, 
    \[
    \| w_{\delta} \|_{C^{1, 1} \left( \overline{B_{\kappa\delta^{-1/N}}} \right)} \le 2 \| w_{\delta} \|_{C^{1, 1} \left( \overline{B_{\kappa\delta^{-1/N}}^+} \right)}
    \]
\end{proof}
Our next aim is to prove the following
\begin{proposition}\label{prop:TangentDerivaticesC11To0}
    For any $r \in (0, r_{2}/4)$,  
    \begin{align}\label{eqn:C1aTangentialConvergence}
        \begin{split}
            \| \nabla w_{\delta} - \left(\nabla w_{\delta} \cdot n\right) n & \|_{C^{1, \theta} \left( \overline{\widetilde D_\delta \cap \left(B_{\kappa\delta^{-1/N}}^+ \setminus B_r \right) } \right)} + \\
            & + \| \nabla w_{\delta} - \left(\nabla w_{\delta} \cdot n\right) n \|_{C^{1, \theta} \left( \overline{ B_{\kappa\delta^{-1/N}}^+ \setminus \left( \widetilde D_\delta \cup B_r \right) } \right)} \to 0
        \end{split}
    \end{align}
    as $\delta \to 0$, where $n=z/\vert z \vert$.
\end{proposition}
\begin{proof}
Let $\IM$, $w$ and $m$ be respectively the principal positive eigenvalue, the optimal eigenfunction and the corresponding weight of the design problem in $\R^N$, as introduced in~\eqref{eq:lim_prob}. Recall that $w$ is radially symmetric. 
Let us define $h_{i}=\partial_i w$ the partial derivatives of $w$
in the directions tangent to $\{ z_N = 0 \}$, i.e. $i = 1, \dots, N-1$. 
Notice that each $h_{i}$ is even
 with respect to $\{ z_N = 0 \}$. The functions $h_i$ solve the transmission problem
\begin{equation}\label{eqn:TransmissionPbLimit}
\begin{cases}
-\Delta h_i = \IM m h_i  & \text{in } B_{r_2} \bigcup \left( B_{\kappa\delta^{-1/N}} \setminus B_{r_2} \right) , \\
[h_i] = 0, \quad [\partial_n h_i] = - \IM (1+\beta) w n_i & \text{on } \partial B_{r_2},
\end{cases}
\end{equation}
where $n=z/|z|$.

Now consider the optimal (reflected) eigenfunctions $w_{\delta}$ in the blow-up scale, and denote $v_{\delta, i} := \partial_i w_{\delta}$, for $i = 1, \dots, N-1$. From \eqref{eqn:TransmissionPbReflected}, we know that they solve the transmission problem
\begin{equation}\label{eqn:TransmissionPbvdeltaReflected}
\begin{cases}
-\diverg\left( A^{\delta} \nabla v_{\delta, i} \right) = \delta^{2/N} \od(\delta) \widetilde m_\delta \left( v_{\delta, i} J_{\delta} + w_{\delta} \partial_i J_{\delta}  \right) + \diverg\left( G_i^{\delta} \right) & \text{in } \widetilde D_\delta \bigcup \left( B_{\kappa\delta^{-1/N}} \setminus \overline{\widetilde D_\delta} \right) , \\
[v_{\delta, i}] = 0, \quad [A^{\delta} \nabla v_{\delta, i} \cdot \nu] = - \delta^{2/N} \od(\delta) (1+\beta) w_{\delta} J_{\delta} \nu_i & \text{on } \partial\widetilde D_\delta .
\end{cases}
\end{equation}
In order to compare the functions $v_{\delta, i}$ and $h_i$ in $B_{\kappa\delta^{-1/N}}$, for $\delta$ sufficiently small we introduce a one-parameter family of radial diffeomorphisms $\Theta_{\delta}:\R^N \to \R^N$, with the following properties:
\begin{align}\label{eqn:PropertiesC1aDiffeoTheta}
    \begin{split}
        &\|\Theta_{\delta} - \Id  \|_{C^{1, \alpha}(\R^N)} = o_\delta(1) \; \text{for every } 0 < \alpha < 1 , \text{ as } \delta \to 0, \quad\\
        & \Theta_{\delta} = \Id \text{ in }\R^N\setminus B_{\frac{3}{2} r_2},\quad\Theta_{\delta}(B_{r_2})= \widetilde{D}_{\delta} , \quad
        \Theta_{\delta}(\partial B_{r_2})= \partial \widetilde D_\delta ,
    \end{split}
\end{align}
where $|B_{r_{2}}|=2$. The existence of such family 
of diffeomorphisms follows from Proposition \ref{pro:nearlysphericalalpha} (actually $\Theta_\delta$ 
depends on $\varphi_\delta$ there); we refer to \cite[Section 6.1]{ferreri_verzini2} for more details on 
such diffeomorphism. In particular, \eqref{eqn:PropertiesC1aDiffeoTheta} holds true with 
$\alpha=\theta$.

Let the functions $\phi_{\delta, i}$ be defined as 
\[
\phi_{\delta, i} := v_{\delta, i} \circ \Theta_{\delta} - h_i .
\]
From \eqref{eqn:TransmissionPbLimit} and \eqref{eqn:PropertiesC1aDiffeoTheta} we deduce that the functions $\phi_{\delta, i}$ solve, in $H^1\left( B_{\kappa\delta^{-1/N}} \right)$, the transmission problem
\begin{equation}\label{eqn:TransmissionPbDifference}
\begin{cases}
-\diverg\left( M^{\delta} \nabla \phi_{\delta, i} \right) = f_i^{\delta} + \diverg\left( F_i^{\delta} \right) & \text{in } B_{r_2} \bigcup \left( B_{\kappa\delta^{-1/N}} \setminus B_{r_2} \right) , \\
[\phi_{\delta,i}] = 0, \quad [M^{\delta} \nabla \phi_{\delta, i} \cdot n] = g_i^{\delta} & \text{on } \partial B_{r_2} ,
\end{cases}
\end{equation}
where we have denoted 
\begin{align*}
 & Y_{\delta} := \vert \det \left( D \Theta_{\delta} \right) \vert , \quad M^{\delta} := Y_{\delta} D\Theta_{\delta}^{-1} A^{\delta} D\Theta_{\delta}^{-T} \qquad \text{in } B_{\kappa\delta^{-1/N}} , \quad Y_{\delta, T} := \vert D\Theta_{\delta}^{-T} n \vert Y_{\delta} \qquad \text{on } \partial B_{r_2}, 
 \\
 & f_i^{\delta} := \delta^{2/N} \od(\delta) \left( \widetilde m_\delta v_{\delta, i}  J_{\delta} + \widetilde m_\delta w_{\delta} \partial_i J_{\delta}   \right) \circ \Theta_{\delta} - \IM m h_i, \quad 
 F_i^{\delta} :=Y_{\delta}D\Theta_{\delta}^{-1} G_i^{\delta} + \left( M^{\delta} - \Id \right) \nabla h_i \qquad \text{in } B_{\kappa\delta^{-1/N}} , 
 \\
 & g_i^{\delta} := (1+\beta) \left( \IM  w n_i - \delta^{2/N} \od(\delta) \left(w_{\delta} J_{\delta} \nu_i\right) \circ \Theta_{\delta} Y_{\delta, T} \right) - \left[ \left( M^{\delta} - Id \right) \nabla h_i \cdot n \right] \qquad \text{on } \partial B_{r_2} .
\end{align*}
Without loss of generality, we can also assume that
\[
\phi_{\delta,i} = 0 \text{ on } \partial B_{\kappa\delta^{-1/N}}.
\]
Indeed, for $\delta$ sufficiently small, it is sufficient to subtract from $\phi_{\delta, i}$ the function $z_{\delta, i}$ solution to
\[
\begin{cases}
    -\diverg\left( M^{\delta} \nabla z_{\delta, i} \right) = 0 & \text{in } B_{\kappa\delta^{-1/N}}, \\
    z_{\delta, i} = \phi_{\delta, i} & \text{on } \partial B_{\kappa\delta^{-1/N}},
\end{cases}
\]
and to notice that, by elliptic regularity and \cite[Remark 4.5]{MPV3},
\[
\| z_{\delta, i} \|_{C^{1, \theta}\left( \overline{B_{\kappa\delta^{-1/N}}} \right)} \to 0, \quad \text{as } \delta \to 0.
\]
We are in the position to apply the results in \cite[Theorem 1.2]{Dong2021:TransmissionProblems}, and since \cite[Lemma 5.4]{ferreri_verzini2}, Proposition \ref{prop:vecchiolavoro} and \eqref{eq:lemmaA1} imply that
\[
\| f_i^{\delta} \|_{L^{\infty}\left( B_{\kappa\delta^{-1/N}} \right)} \to 0, \quad \| F_i^{\delta} \|_{C^{0, \theta}\left( B_{\kappa\delta^{-1/N}} \right)} \to 0 \quad \text{and} \quad \| g_i^{\delta} \|_{C^{0, \theta}\left( \partial B_{r_2} \right)} \to 0 
\]
as $\delta \to 0$, \cite[Theorem 1.2]{Dong2021:TransmissionProblems} gives that, for all $i = 1, \dots, N-1$
\begin{equation}\label{eqn:C1aSeparateConvergeTangentialDerivatives}
\| \phi_{\delta, i} \|_{C^{1, \theta}\left( \overline{B_{r_2}} \right)} + \| \phi_{\delta, i} \|_{C^{1, \theta}\left( \overline{B_{\kappa\delta^{-1/N}} \setminus B_{r_2}} \right)} \to 0, \qquad \text{as } \delta\to 0.
\end{equation}
Moreover, as usual isolating $\partial_N^2 w_{\delta}$ in equation \eqref{eq:eq_in_forma_di_div} and composing it with the diffeomorphism $\Theta_{\delta}$, it can be seen using \eqref{eqn:C1aSeparateConvergeTangentialDerivatives} that 
\[
\| \partial_N \left( v_{\delta, N} \circ \Theta_{\delta}  \right) - \partial_N^2 w \|_{C^{0, \theta}\left( \overline{B_{r_2} \cap B_{\kappa\delta^{-1/N}}^+} \right)} + \| \partial_N \left( v_{\delta, N} \circ \Theta_{\delta}  \right) - \partial_N^2 w \|_{C^{0, \theta}\left( \overline{B_{\kappa\delta^{-1/N}}^+ \setminus B_{r_2}} \right)} \to 0, \quad \text{as } \delta\to 0.
\]
Hence, we have just shown that
\begin{equation}\label{eqn:C1aSeparateConvergeDerivatives}
\| \phi_{\delta, i} \|_{C^{1, \theta}\left( \overline{B_{r_2} \cap B_{\kappa\delta^{-1/N}}^+} \right)} + \| \phi_{\delta, i} \|_{C^{1, \theta}\left( \overline{B_{\kappa\delta^{-1/N}}^+ \setminus B_{r_2}} \right)} \to 0, \qquad \text{as } \delta\to 0, \qquad \forall \, i = 1, \dots, N.
\end{equation}
Now, in order to prove the proposition, it is sufficient to notice that, by \eqref{eqn:C1aSeparateConvergeDerivatives} and \eqref{eqn:PropertiesC1aDiffeoTheta} we have
\begin{align}\label{eqn:C1aTangentialConvergenceReason}
\begin{split}
    & \| \left(\nabla w_{\delta} - \left(\nabla w_{\delta} \cdot n\right) n \right) \circ \Theta_{\delta} - \left(\nabla w - \left(\nabla w \cdot n\right) n \right) \|_{C^{1, \theta} \left( \overline{\widetilde D_\delta \cap \left(B_{\kappa\delta^{-1/N}}^+ \setminus B_r \right) } \right)} +  \\
    + & \| \left(\nabla w_{\delta} - \left(\nabla w_{\delta} \cdot n\right) n \right) \circ \Theta_{\delta} - \left(\nabla w - \left(\nabla w \cdot n\right) n \right) \|_{C^{1, \theta} \left( \overline{ B_{\kappa\delta^{-1/N}}^+ \setminus \left( \widetilde D_\delta \cup B_r \right) } \right)} \to 0.
\end{split}
\end{align}
as $\delta \to 0$, but since the function $w$ is radial, in $B_{\kappa\delta^{-1/N}}^+$
\[
\nabla w - \left(\nabla w \cdot n\right) n = 0,
\]
so that \eqref{eqn:C1aTangentialConvergenceReason} can be rewritten as
\begin{align}\label{eqn:C1aTangentialConvergenceComposed}
\begin{split}
    \| \left(\nabla w_{\delta} - \left(\nabla w_{\delta} \cdot n\right) n \right) & \circ \Theta_{\delta} \|_{C^{1, \theta} \left( \overline{\widetilde D_\delta \cap \left(B_{\kappa\delta^{-1/N}}^+ \setminus B_r \right) } \right)} + \\
    & + \| \left(\nabla w_{\delta} - \left(\nabla w_{\delta} \cdot n\right) n \right) \circ \Theta_{\delta} \|_{C^{1, \theta} \left( \overline{ B_{\kappa\delta^{-1/N}}^+ \setminus \left( \widetilde D_\delta \cup B_r \right) } \right)} \to 0 .
\end{split}
\end{align}
To conclude, the claim \eqref{eqn:C1aTangentialConvergence} follows from \eqref{eqn:C1aTangentialConvergenceComposed} just rewriting
\[
\nabla w_{\delta} - \left(\nabla w_{\delta} \cdot n\right) n = \left(\nabla w_{\delta} - \left(\nabla w_{\delta} \cdot n\right) n \right) \circ \Theta_{\delta} \circ \Theta_{\delta}^{-1} 
\]
and using \eqref{eqn:PropertiesC1aDiffeoTheta} once again.
\end{proof}
As a consequence of Proposition \ref{prop:TangentDerivaticesC11To0}, we have the following corollary.
\begin{corollary}\label{crl:TangentDerivaticesC01To0}
   For any $r \in (0, r_{2}/4)$,  
    \[
    \| \nabla w_{\delta} - \left(\nabla w_{\delta} \cdot n\right) n \|_{C^{0, 1} \left( \overline{B_{\kappa\delta^{-1/N}} \setminus B_r } \right)} \to 0 , \qquad \text{as $\delta\to 0$},
    \]
    where $n=z/\vert z \vert$.
\end{corollary}

Finally, we can conclude this section with the following improvement of  Proposition \ref{pro:nearlysphericalalpha}.
\begin{proposition}\label{pro:nearlyspherical}
For $\delta$ sufficiently small $\widetilde{D}_{\delta}$ is nearly spherical of class $C^{1,1}$,  parametrized by $\varphi_{\delta}$ (see Definition \ref{def:nearlysph}). In addition
\[
\|\varphi_{\delta}\|_{C^{1,1}(\mathbb{S}^{N-1})}\to 0, \qquad \text{as $\delta\to 0$}.
\]
\end{proposition}
\begin{proof}
In view of  Proposition \ref{prop:TransmissionPb}  and Corollary  \ref{crl:TangentDerivaticesC01To0}, 
the result can be proved by arguing as in \cite[Proposition 5.10]{ferreri_verzini2}.
\end{proof}

\begin{remark}\label{rmk:piùregolare}
The above $C^{1,1}$ decay is sharp, on $\sphere^{N-1}$, because of the reflection of the $N$-th 
derivative. Actually, since $\partial\Omega \in C^{3,\theta}$, one may expect further regularity of the 
free boundary. As a matter of fact, the estimates in $C^{2,\theta}$ of $w_\delta$ hold up to the free 
boundary, from inside and outside; therefore a continuation argument should allow to improve the 
above $C^{1,1}(\sphere^{N-1})$ decay to a $C^{2,\theta}(\sphere^{N-1}_+)$ one, but we do not pursue 
this argument here.
\end{remark}

\section{Quantitative estimates}\label{sec:quantitative}

In this section we conclude the proof of Theorem \ref{thm:nearlyspher1} by obtaining the 
quantitative estimates for the nearly spherical parametrizations of the optimal sets 
$\widetilde{D}_{\delta}$ (in the blow-up scale, as introduced in Sections~\ref{sec:boundbelow}, 
\ref{sec:polar}) and $D_\delta$ (in the original reference). 
These estimates are based on Proposition \ref{pro:nearlyspherical} and on the application 
of the following result.

\begin{theorem}[{\cite[Theorem 1.4]{ferreri_verzini2}}]\label{thm:teoFV}
There exist positive constants $C,\eps$ such that, for all $C^{1,1}$ nearly spherical sets 
$A\subset\R^N$, centered at the origin and parametrized by $\varphi$ satisfying
\begin{enumerate}
\item $\bari(A)=0$,
\item $|A|=2$,
\item $\|\varphi\|_{C^{1,1}(\mathbb{S}^{N-1})}\leq \eps$,
\end{enumerate}
it holds 
\[
\lambda(A,\R^{N})-\lambda(B_{r_2},\R^{N})\geq C \|\varphi\|^2_{L^2(\mathbb{S}^{N-1})}.
\]
\end{theorem}

To apply this result to $\widetilde{D}_{\delta}$ we notice that its first assumption is verified 
by centering the blow-up analysis at $Q_\delta:=\Phi_\delta(\bari(\widetilde{D}_{\delta}))$, 
rather than $P_\delta$ (recall that we are working in this setting since \eqref{eq:cambiotrasla}, thanks to Lemma \ref{lem:bari}). 
On the other hand, $|\widetilde{D}_{\delta}|$ is equal to $2$ only in the limit, thus we need to take into account an error term.

\begin{lemma}\label{lem:quant1}
There exists $C>0$ such that, for $\delta$ sufficiently small,
\[
\left(\frac{|\widetilde D_\delta|}{2}\right)^{2/N}\lambda(\widetilde D_\delta,\R^N)-\IM
\geq C \|\vfi_\delta\|^2_{L^2(\mathbb{S}^{N-1})} + o(\delta^{1/N}).
\]
\end{lemma}
\begin{proof}
First, recall that $\widetilde D_\delta$ is $C^{1,1}$ nearly spherical, with parametrization 
$\varphi_\delta$ satisfying Proposition \ref{pro:nearlyspherical}. We infer that the set 
\[
A_\delta=\frac{2^{1/N}\widetilde{D}_{\delta}}{|\widetilde{D}_{\delta}|^{1/N}}
\] 
is nearly spherical, too:
indeed, for every $\eta \in \partial A_\delta$ there exists $z\in \partial\widetilde{D}_{\delta}$ such that
$\eta=2^{1/N}z|\widetilde{D}_{\delta}|^{-1/N}$, so that 
\begin{equation}\label{eq:Anearly}
\partial A_\delta=\left\{\eta=\left(r_2+\sigma_{\delta}(\theta)\right)\theta\right\},\quad 
\text{with }  
\sigma_{\delta}=\frac{2^{1/N}r_{2}}{|\widetilde{D}_{\delta}|^{1/N}} - r_2+\frac{2^{1/N}}{|\widetilde{D}_{\delta}|^{1/N}}\vfi_{\delta}.
\end{equation}
We are in position to apply Theorem \ref{thm:teoFV} to $A_\delta$, obtaining 
\[
\left(\frac{|\widetilde D_\delta|}{2}\right)^{2/N}\lambda(\widetilde D_\delta,\R^N)-\IM\geq C \|\sigma_\delta\|^2_{L^2}.
\]
In turn, using Lemma \ref{le:ImD},
\[
\begin{split}
\|\sigma_\delta\|^2_{L^2} &= \left\|\frac{2^{1/N}r_{2}}{|\widetilde{D}_{\delta}|^{1/N}} - r_2+\frac{2^{1/N}}{|\widetilde{D}_{\delta}|^{1/N}}\vfi_{\delta} \right\|^2_{L^2}
\ge \frac12\left(\frac{2}{|\widetilde{D}_{\delta}|}\right)^{2/N}\left\|\vfi_{\delta}\right\|^2_{L^2}
-\left\|\left(\frac{2}{|\widetilde{D}_{\delta}|}\right)^{1/N} -  1 \right\|^2_{L^2}r_2^2\\
&\ge C(1-\delta^{1/N})\left\|\vfi_{\delta}\right\|^2_{L^2} + o(\delta^{1/N}) = 
C\left\|\vfi_{\delta}\right\|^2_{L^2} + o(\delta^{1/N}),
\end{split}
\]
where we used also Proposition \ref{pro:nearlyspherical}, and the claim follows.
\end{proof}
\begin{proposition}
As $\delta\to 0$,
\begin{equation}\label{eq:quantphi}
\|\varphi_\delta\|^2_{L^2(\mathbb{S}^{N-1})}=o(\delta^{1/N}).
\end{equation}
\end{proposition}
\begin{proof}
We apply the second conclusion of  Theorem~\ref{thm:sviluppoesatto}, Remark \ref{rmk:scorciatoia} and 
Lemma \ref{lem:quant1} to obtain
\begin{equation}\label{eq:stima1}
\begin{split}
-\IM\Gamma \widehat H \delta^{1/N}+o(\delta^{1/N})
&= \delta^{2/N}\od(\delta)-\IM \\
&= \delta^{2/N}\od(\delta) - \left(\frac{|\widetilde D_\delta|}{2}\right)^{2/N}\lambda(\widetilde D_\delta,\R^N) + \left(\frac{|\widetilde D_\delta|}{2}\right)^{2/N}\lambda(\widetilde D_\delta,\R^N) -\IM
\\
&\ge  -\IM\Gamma\widehat H\delta^{1/N}+ C \|\vfi_\delta\|^2_{L^2} +o(\delta^{1/N}),
\end{split}
\end{equation}
implying the claim.
\end{proof}

Finally, we bring back the quantitative information~\eqref{eq:quantphi} to $D_\delta$.
\begin{proof}[Proof of Theorem \ref{thm:nearlyspher1}]
By now we have obtained 
 \[
\widetilde D_\delta=\left\{z \in B_{k\delta^{-1/N}} : |z|<r_2+\varphi_\delta\left(\frac{z}{|z|}\right)\right\},
\]
so that
\begin{equation}\label{eq:Dtildedelta+}
\begin{split}
\widetilde D_\delta^+&:=\left\{z\in B^{+}_{k\delta^{-1/N}} : |z|< r_2+\varphi_\delta
\left(\frac{z}{|z|} \right)\right\} = \left\{z\in B^{+}_{2r_2} : |z|< r_2+\varphi_\delta
\left(\frac{z}{|z|} \right)\right\}
\end{split}
\end{equation}
for $\delta$ small enough.

On the other hand, recalling \eqref{eq:Dtilde},
\[
\widetilde D^+_\delta=\left\{z\in B^{+}_{2r_2}  : z=\tfrac{\Psi_\delta(x)}{\delta^{1/N}},
\text{ for some $x\in D_\delta$}\right\}.
\]
Then $x\in\partial D_\delta$ if and only if $z \in \partial\widetilde D^+_\delta$, and the theorem 
will follow by a change of variable.

Thus, in the following, we consider a generic $z\in B_{2r_2}^+$ and $x =\Phi_\delta(\delta^{1/N}z)$, 
so that $|x|=O(\delta^{1/N})$ as $\delta\to0$. We obtain
\begin{equation}\label{eq:rdelta}
x =\Phi_\delta(\delta^{1/N}z) = \delta^{1/N}z + \delta^{2/N}R_\delta(z),
\end{equation}
where the reminder $R_\delta$ is uniformly bounded in $C^{1,1}$: notice that this only requires
$\Phi_\delta\in C^{1,1}$ uniformly in $\delta$, i.e. $\partial\Omega\in C^{2,1}$ 
(recall Remark \ref{rmk:unoinmeno}). Indeed, we have
\[
\begin{split}
\delta^{2/N}R_\delta(z) &= \Phi_\delta(\delta^{1/N}z) - \delta^{1/N}z = 
\Phi_\delta(\delta^{1/N}z) - \Phi_\delta(0) - D\Phi_\delta(0)\delta^{1/N}z,\\
\delta^{1/N}DR_\delta(z) &= D\Phi_\delta(\delta^{1/N}z)- D\Phi_\delta(0),
\end{split}
\]
which yield
\[
\begin{split}
\delta^{2/N}|R_\delta(z)| &= \left|\int_0^1 \left[\frac{d}{dt} \Phi_\delta(t\delta^{1/N}z) - D\Phi_\delta(0)\delta^{1/N}z\right]\,dt\right| \\
&\le \delta^{1/N}|z|\int_0^1 \left|D\Phi_\delta(t\delta^{1/N}z) - D\Phi_\delta(0)\right|\,dt\\
&\le \delta^{2/N}|z|^2 \cdot \frac12 \|D\Phi_\delta\|_{C^{0,1}}\le \delta^{2/N} \cdot 2 \|D\Phi_\delta\|_{C^{0,1}}r_2^2
\end{split}
\]
and
\[
\delta^{1/N}\left|DR_\delta(z_1)-DR_\delta(z_2)\right| = 
\left|D\Phi_\delta(\delta^{1/N}z_1) - D\Phi_\delta(\delta^{1/N}z_2)\right| 
\le \delta^{1/N}\|D\Phi_\delta\|_{C^{0,1}}|z_1-z_2|.
\]

From \eqref{eq:rdelta} we infer
\begin{equation}\label{eq:modconres}
|x|^2 = \delta^{2/N}|z|^2 + \delta^{\frac3N}\tilde R_\delta(z),
\end{equation}
where $\tilde R_\delta$ is  uniformly in $C^{1,1}$, too.

Let us introduce the polar coordinates 
\[
\rho=|x|,\ \vartheta=\frac{x}{|x|},\qquad r=|z|,\ \xi=\frac{z}{|z|}.
\]
Taking $x \in \partial D_\delta \cap \Omega$, and the corresponding 
$z=r\xi\in \partial \widetilde D^+_\delta$, then $r = r_2+\varphi_\delta(\xi)$. 
By the properties of $\Phi_\delta$, $\Psi_\delta$, $\varphi_\delta$ we have that the map
\[
\sphere^{N-1}\ni \xi \mapsto \frac{\Phi_\delta(\delta^{1/N}(r_2+\varphi_\delta(\xi))\xi)}{|\Phi_\delta(\delta^{1/N}(r_2+\varphi_\delta(\xi))\xi))|} 
= \vartheta_\delta(\xi) \in \sphere^{N-1}
\]
is uniformly bounded in $C^{1,1}$, with inverse $\vartheta \mapsto \xi=\xi_\delta(\vartheta)$, for $\delta$ 
small enough.  Substituting 
into \eqref{eq:modconres} we infer 
that \eqref{eq:defrho} holds true with $\rho_\delta=\rho_\delta(\vartheta)$ implicitly defined by 
\[
2r_2\rho_\delta+\rho^{2}_{\delta}=2r_2\varphi_\delta(\xi_\delta(\vartheta))+
\varphi_\delta^2(\xi_\delta(\vartheta))+\delta^{1/N}\tilde 
R_\delta\Big((r_2+\varphi_\delta(\xi_\delta(\vartheta)))\xi_\delta(\vartheta)\Big).
\]
Then 
\[
\rho_\delta(\vartheta) = \varphi_\delta(\xi_\delta(\vartheta))+\delta^{1/N}Z_\delta(\vartheta),
\]
where $Z_\delta$ is uniformly bounded in $C^{1,1}$. Recalling \eqref{eq:quantphi} and Proposition 
\ref{pro:nearlyspherical} we obtain the desired estimates.
\end{proof}

\begin{remark}\label{rmk:GagliardoNirenberg}
From Theorem~\ref{thm:nearlyspher1}, thanks to the Gagliardo Nirenberg interpolation inequalities, one can deduce as in~\cite[Corollary~1.5]{ferreri_verzini2} estimates on different norms of $\rho_\delta$, the most interesting one being
\[
\|\rho_\delta\|_{C^{1,\alpha}}=o\left(\delta^{\frac{(1-\alpha)}{N(4+N)}}\right) \quad \forall\, \alpha\in (0,1).
\]
\end{remark}

\bigskip

\textbf{Acknowledgments.} We warmly thank Manuel del Pino for pointing our attention on 
the techniques introduced in \cite{MR1736974,delpino_flores}.

Work partially supported by: the Italian MUR and the European Union through the 
PRIN projects  
P20227HX33Z ``Pattern formation in nonlinear phenomena'' (BP and GV), 
``$NO^3$'' P2022R537CS, CUP\_F53D23002810006 (DM) and 
``Mathematics for Industry 4.0'' P2020F3NCPX, CUP\_F19J21017440001 (DM); 
the European Research Council, under the EU's Horizon 2020 research and innovation program, 
through the project ERC VAREG - ``Variational approach to the regularity of the free boundaries'' (No. 
853404) (LF); the Portuguese government through FCT/Portugal
under the project PTDC/MAT-PUR/1788/2020 (GV);  the project "Start" within the program of the University "Luigi Vanvitelli" reserved to young researchers (BP).

LF acknowledges the MIUR Excellence Department Project awarded to the Department of Mathematics, University of Pisa, CUP I57G22000700001.

All authors have been partially supported and/or are members of the INdAM-GNAMPA group.

\bigskip

\textbf{Disclosure statement.} The authors report there are no competing interests to declare.


\bibliography{frac_eig}{}
\bibliographystyle{abbrv}
\medskip
\small
\begin{flushright}
\noindent \verb"lorenzo.ferreri@sns.it"\\
Classe di Scienze, Scuola Normale Superiore, \\
Piazza dei Cavalieri 7, 56126 Pisa, Italy\\
\medskip
\noindent \verb"dario.mazzoleni@unipv.it"\\
Dipartimento di Matematica  ``F. Casorati'', \\
Universit\`a di Pavia\\
via Ferrata 5, 27100 Pavia, Italy\\
\medskip
\noindent \verb"benedetta.pellacci@unicampania.it"\\
Dipartimento di Matematica e Fisica,\\
Universit\`a della Campania  ``Luigi Vanvitelli''\\
viale A. Lincoln 5, 81100 Caserta, Italy\\
\medskip
\noindent \verb"gianmaria.verzini@polimi.it"\\
Dipartimento di Matematica, Politecnico di Milano\\
piazza Leonardo da Vinci 32, 20133 Milano, Italy\\
\end{flushright}

\end{document}